\documentclass[11pt,a4paper,onecolumn,notitlepage]{article}

\usepackage{combelow} %for romanian s/t-comma
\usepackage{cmap} % fix search and cut-and-paste in Acroread
\usepackage{CJKutf8}
\usepackage{setspace}
\usepackage[pdfpagemode=None,
pagebackref=true,
bookmarksopen=false,
hyperindex=true,
colorlinks=true,
linkcolor=blue,
citecolor=blue,
pdftitle={Four negations and the spectral presheaf},
pdfauthor={Benjamin Engel, Ryshard-Pavel Kostecki}]{hyperref}
\usepackage{tabularx}
\usepackage{geometry}
\usepackage[X2,T2A,T1]{fontenc}
\usepackage{setspace}

\usepackage{bibspacing}
\usepackage{enumitem}
\renewcommand{\footref}[1]{\textsuperscript{\ref{#1}}}
\usepackage{mathptmx}
\usepackage{bbding}
\usepackage{amssymb} %\mho and...?
\usepackage{amsfonts}
\usepackage{amsthm}
\usepackage{amsmath} % \DeclareMathOperator*
\usepackage{graphicx} %for roundrule
\usepackage{calc} %for roundrule
\usepackage{tikz} %for coimp

\usepackage{roundrule}

\newcommand{\op}{\mathrm{op}}
\DeclareFontFamily{U}{mathb}{\hyphenchar\font45}
\DeclareFontShape{U}{mathb}{m}{n}{
      <5> <6> <7> <8> <9> <10>
      <10.95> <12> <14.4> <17.28> <20.74> <24.88>
      mathb10
      }{}
\DeclareSymbolFont{mathb}{U}{mathb}{m}{n}
\DeclareMathSymbol{\sqbullet}{1}{mathb}{"0D}

\DeclareSymbolFont{stmry}{U}{stmry}{m}{n}
\DeclareMathSymbol\rightarrowtriangle\mathop{stmry}{"5F}
\DeclareMathSymbol\leftarrowtriangle\mathop{stmry}{"5E}

\newcommand{\Proj}{\mathrm{Proj}} %projections in the algebra

\newcommand{\stone}{\varkappa}
\newcommand{\minbool}[1]{{\langle{#1}\rangle}}
\makeatletter
\newcommand{\superimpose}[3][\mathord]{#1{\mathpalette\superimpose@{{#2}{#3}}}}
\newcommand{\superimpose@}[2]{\superimpose@@{#1}#2}
\newcommand{\superimpose@@}[3]{%
  \ooalign{%
    \hfil$\m@th#1#2$\hfil\cr
    \hfil$\m@th#1#3$\hfil\cr
  }%
}
\makeatother

\newcommand{\orlov}{\boxtimes}%{\diamond}

\newcommand{\Subclop}{\Sub_{\mathrm{clop}}}
\newcommand{\himp}{\imp}
\newcommand{\cohimp}{\raisebox{0.1121ex}{\,{\rotatebox[origin=c]{180}{$\coimp$}}\,}}
\newcommand{\hnot}{\lnot}
\newcommand{\cohnot}{\reflectbox{$\lnot$}}

\DeclareSymbolFont{TXsymbols}{OMS}{txsy}{m}{n}
\SetSymbolFont{TXsymbols}{bold}{OMS}{txsy}{b}{n}
\DeclareMathSymbol{\imp}{\mathrel}{TXsymbols}{33}
\newcommand{\coimp}{\mathrel{%
  \hspace{.1ex}
  \begin{tikzpicture}[baseline=-.57ex, line width=.130ex]
    \draw[-] (-0.1ex,0) --(1.5ex,0);
    \draw[-, line width=.01ex, fill=black]
             (1.35ex,0) -- (1.84ex, .48ex)
                       -- (1.91ex ,.418ex)
                       -- (1.55ex,   0ex)
                       -- (1.91ex ,-.418ex)
                       -- (1.84ex,-.48ex)
                       -- (1.35ex,0ex);
  \end{tikzpicture}
\hspace{.1ex}}}

\newcommand{\coimpscaledi}[1]{\mathrel{%
  \hspace{.1ex}
  \begin{tikzpicture}[baseline=-0.5ex, line width=.130ex,scale=#1]
    \draw[-, line width=.096ex, fill=black] (-0.1ex,0) --(1.5ex,0);
    \draw[-, line width=.01ex, fill=black]
             (1.35ex,0) -- (1.84ex, .48ex)
                       -- (1.91ex ,.418ex)
                       -- (1.55ex,   0ex)
                       -- (1.91ex ,-.418ex)
                       -- (1.84ex,-.48ex)
                       -- (1.35ex,0ex);
  \end{tikzpicture}
\hspace{.1ex}}}

\newcommand{\cohimpi}{\raisebox{0.071ex}{{\,\rotatebox[origin=c]{180}{$\coimpscaledi{0.75}$}\,}}}

\newcommand{\Sub}{\mathrm{Sub}}

\renewcommand{\mid}{\,\st}
\newcommand{\st}{\boldsymbol{:}\,}

\renewcommand{\H}{\mathcal{H}} % Hilbert space
\newcommand{\A}{\mathcal{A}}   % C*-algebra
\newcommand{\N}{\mathcal{N}}   % W*-algebra
\newcommand{\CC}{\mathbb{C}} % boolean complex numbers
\newcommand{\BBB}{\mathfrak{B}}     %prefix for `bounded'       %=\BB
\newcommand{\BH}{{\BBB(\H)}}           %bounded lin. ops on H.sp.
\newcommand{\ra}{\rightarrow}      %right arrow
\newcommand{\limp}{\Rightarrow}    %logical implication
\newcommand{\iso}{\cong}           %isomorphism 
\renewcommand{\sp}{\mathrm{sp}}   %spectrum
\newcommand{\Ob}{\mathrm{Ob}}     %objects of a category
\newcommand{\Mor}{\mathrm{Mor}}   %morphisms of a category
\newcommand*{\catname}[1]{{\mathtt{#1}}} %changed from texttt to mathtt: 9.6.22
\newcommand{\id}{\mathrm{id}}     %identity mapping
\newcommand{\Set}{\catname{Set}} 
\newcommand*{\cytat}[1]{\guillemotleft{#1}\guillemotright}

\newtheorem{lemma}{Lemma}[section]

\newtheorem{proposition}[lemma]{Proposition}
\newtheorem{definition}[lemma]{Definition}
\newtheorem{corollary}[lemma]{Corollary}

	\theoremstyle{definition}%
	\newtheorem{remark}[lemma]{Remark}%
\usepackage{datetime}
\usepackage[2cell,all,cmtip]{xy}

\usepackage{framed}
\usepackage{sectsty}
\allsectionsfont{\normalsize}

\renewcommand*{\backref}[1]{}
\renewcommand*{\backrefalt}[4]{%
  \ifcase #1 %
    \relax
  \or
    \nobreak{$\uparrow$~{#2}}.
  \else
    \nobreak{$\uparrow$~{#2}}.
  \fi%
}
\newcommand{\sps}{\sp_{\mathrm{S}}}

\newcommand{\Chi}{\chi}%{\raisebox{0.4ex}{\scalebox{1.2}{$\chi$}}}
\newcommand{\coloneqq}{:=}

\newcommand{\context}{{\mathcal{V}_\mathrm{c}(L)}}
\newcommand{\contextnotcomplete}{{{\mathcal{V}(L)}}}

\newcommand{\extlattice}{K}

\newcommand{\bigcurlywedge}{\raisebox{-0.4ex}{\scalebox{1.6}{$\curlywedge$}}}

\newcommand{\evaintlattice}[1]{{#1}_{{\eva\eva}}}
\newcommand{\evaintlatticeext}{\evaintlattice{K}}
\newcommand{\coevaintlattice}[1]{{#1}_{{\coeva\coeva}}}
\newcommand{\coevaintlatticeext}{\coevaintlattice{K}}

\renewcommand{\bot}{{\pmb{\perp}}}
\renewcommand{\top}{{\rotatebox[origin=c]{180}{$\bot$}}}

\newcommand{\topi}{{\scalebox{0.75}{\top}}}
\newcommand{\boti}{{\scalebox{0.75}{$\bot$}}}

\makeatletter
\providecommand{\leftsquigarrow}{%
  \mathrel{\mathpalette\reflect@squig\relax}%
}
\newcommand{\reflect@squig}[2]{%
  \reflectbox{$\m@th#1\rightsquigarrow$}%
}
\makeatother

\newcommand{\oepsilonequc}[1]{\lceil{#1}\rceil_{\oepsilon}}
\newcommand{\iepsilonequc}[1]{\lceil{#1}\rceil_{\iepsilon}}

\newcommand{\upint}{\eva\eva}
\newcommand{\coupint}{\coeva\coeva}

\newcommand{\eva}{\bullet}
\newcommand{\coeva}{\circ}
\renewcommand{\neg}{\dagger}
\newcommand{\evai}{{\scalebox{0.70}{$\eva$}}}
\newcommand{\coevai}{{\scalebox{0.70}{$\coeva$}}}

\newcommand{\cqieva}{{\protect\superimpose{\raisebox{-0.5ex}{{\evai}}}{\sim}}}
\newcommand{\qieva}{{\protect\superimpose{\raisebox{-0.5ex}{{\coevai}}}{\sim}}}

\newcommand{\cqineg}{{\scalebox{0.85}{\raisebox{0.1ex}{$\cqieva$}}}}
\newcommand{\qineg}{{\scalebox{0.85}{\raisebox{0.1ex}{$\qieva$}}}}

\newcommand{\odelta}{{\delta^{\land}}} % replaces \delta; outer daseinisation
\newcommand{\odeltabar}{{{\bar{\delta}}^{\land}}} % replaces \indelta
\newcommand{\oepsilon}{{\varepsilon^{\land}}} % replaces \varepsilon
\newcommand{\idelta}{{\delta^{\lor}}} % inner daseinisation
\newcommand{\ideltabar}{{{\bar{\delta}}^{\lor}}}
\newcommand{\iepsilon}{{\varepsilon^{\lor}}}

\newcommand{\gentzen}{\vdash}
\newcommand{\cogentzen}{\dashv}
\newcommand{\iffgentzen}{\cogentzen\gentzen}

\newcommand{\Frm}{\mathrm{Frm}}
\newcommand{\Seq}{\mathrm{Seq}}
\newcommand{\Int}{\mathbf{Int}}
\newcommand{\coInt}{\mathbf{coInt}}
\newcommand{\biInt}{\mathbf{biInt}}
\newcommand{\QInt}{\mathbf{QInt}}
\newcommand{\coQInt}{\mathbf{coQInt}}
\newcommand{\biQInt}{\mathbf{biQInt}}

\newcommand{\biVak}{\mathbf{biVak}}
\newcommand{\Vak}{\mathbf{Vak}}

\newcommand{\igequ}[1]{{\lfloor{#1}\rfloor}}

\renewcommand{\L}{\mathcal{L}}

\newcommand{\subclopevaeva}{{(\Subclop(\Sigma^L))}_{\eva\eva}}
\newcommand{\subclopcoevacoeva}{{(\Subclop(\Sigma^L))}_{\coeva\coeva}}

\begin{document}
\setlength{\belowdisplayskip}{3pt} \setlength{\belowdisplayshortskip}{3pt}
\setlength{\abovedisplayskip}{2pt} \setlength{\abovedisplayshortskip}{2pt}

\begin{center}
{\Large\textbf{Four negations and the spectral presheaf}}\\
\vspace{0.2cm}
\hspace{-0.5cm}\begin{minipage}{.25\textwidth}
\centering
\textit{\;\;\;Benjamin Engel $^\mathrm{a)}$}\\
\vspace{-0.15cm}{\tiny\texttt{benjamin.engel@chello.at}}
\end{minipage}%
\begin{minipage}{.25\textwidth}
\centering
\textit{\;\;Ryshard-Pavel Kostecki $^\mathrm{b)}$}\\
\vspace{-0.15cm}{\tiny\texttt{kostecki@fuw.edu.pl}}
\end{minipage}\\
%\begin{center}
\vspace{0.1cm}
{\footnotesize $^\mathrm{a)}$ Institut f\"{u}r Mathematik, Universit\"{a}t Innsbruck, Technikerstra\ss{}e 13, 6020 Innsbruck, Austria\\$^\mathrm{b)}$ Centrum pre V\'{y}skum Kvantovej Inform\'{a}cie, Slovensk\'{a} Akad\'{e}mia Vied, D\'{u}bravsk\'{a} cesta 9, 84511 Bratislava, Slovakia\\\vspace{0.1cm}\ }{\scriptsize 11 March 2026}\footnote{Except of Section \ref{section.erratum}, which is dated: 12 August 2026.}\\\vspace{0.4cm}\textit{dedicated to Micha{\l} Heller on the occasion of his 90th birthday}\vspace{0.2cm}
\end{center}

\begin{abstract}
\noindent Using Vakarelov's theory of lattice logics with negation, we introduce the (co)quasiintuitionistic logic, and prove its soundness and completeness with respect to the class of (co)quasiintuitionistic algebras. Combining these algebras together, we obtain biquasiintuitionistic algebras and the biquasiintuitionistic logic. Their further extension with the Skolem algebra structure defines Akchurin algebras and the respective logic, which is a product of biquasiintuitionistic and biintuitionistic logics, featuring four distinct negations. Next we generalise the framework of spectral presheaves (which is a main object in the Butterfield--Isham--D\"{o}ring topos theoretic approach to quantum mechanics) to arbitrary complete orthocomplemented lattices, and show that the orthocomplementation determines two negation operators on the spectral presheaf (one paraconsistent, another paracomplete), equipping the set of all closed-and-open subpresheaves of a spectral presheaf with the structure of a biquasiintuitionistic algebra. Combined with the generic Skolem (i.e. Heyting and Brouwer) algebra structure of this set, this gives a particular instance of an Akchurin algebra. We also show that the underlying orthocomplemented lattice can be reconstructed as an internal object of the spectral presheaf, resulting as the image of a double coquasiintuitionistic (resp., quasiintuitionistic) negation monad (resp., comonad). Finally, we prove a no-go theorem for the claim that the spectral presheaf is a model of a dialectical (or any other) relevance logic.
\end{abstract}

\setcounter{section}{-1}
\section{Erratum\label{section.erratum}}

\begin{remark}
Proposition \ref{prop.inner.daseinisation.is.not.clopen} implies that Definition \ref{def: Inner daseinisation} essentially trivialises, so Propositions \ref{prop: Inner Daseinisation} and \ref{thm: idelta adjunction} are inapplicable, and, as a result, Lemma \ref{lemma.subclop.coeva.properties} does not hold. Thus, the part of results of the current paper that depends on Lemma \ref{lemma.subclop.coeva.properties}, is not true under current formulation. We have worked out a solution to this problem, but it requires some additional structures to be defined. It will be posted soon in the next update of the current preprint.
\end{remark}

\begin{proposition}\label{prop.inner.daseinisation.is.not.clopen}
Let $a\in L$. Consider the  functor $\idelta(a):(\context)^{\op} \ra \Set,$ acting by $V\mapsto (\idelta(a))_{V}\coloneqq \stone_{V}\left(\ideltabar_{V}(a)\right)$ $\forall V\in \mathrm{Ob}(\context)$, where
\begin{equation}
\ideltabar_{V}: L \ni a \mapsto \ideltabar_{V}(a)=\bigvee\lbrace b\in V\mid b\leq a\rbrace \in V\;\;\forall V\in \context,
\label{eqn.inner.daseinisation}
\end{equation}
and acting by $(V\subseteq W) \mapsto \cdot|_{V}$ on morphisms of $\context$, where $|_{V}:\sps(W) \ra \sps(V)$ is the restriction map $\lambda\mapsto\lambda|_{V}$ $\forall\lambda\in\sps(W)$. Then $\idelta(a)$ is a closed-and-open subpresheaf of $\Sigma^{L}$ $\forall a\in L$ if{}f ($\context$ is minimal or completely disjoint).
\end{proposition}

\begin{proof}
Assume that $\idelta$ is a presheaf and $\context$ is neither minimal nor completely disjoint. Then $\exists a,b\in L\setminus\{0,1\}$ $\exists V\in \context\st a\notin \minbool{b}$, $a,b\in V$, $|V| = 8$,  and $a$ is comparable to $b$ or $b^{\perp}$. Assume that, without loss of generality, $a$ is comparable to $b$ and $a<b$. Then $(\idelta(a))_{V} = \stone_{V}(a) \neq \varnothing$ and $(\idelta(a))_{\minbool{b}} = \stone_{\minbool{b}}(0) = \varnothing$. Since $\idelta(a)$ is assumed to be a closed-and-open subpresheaf of $\Sigma^{L}$, and $\minbool{b} \subseteq V$, it must hold that $((\idelta(a))_{V})\vert_{\minbool{b}} \subseteq (\idelta(a))_{\minbool{b}} = \varnothing$, leading to a contradiction.  Assume that $\context$ is minimal or completely disjoint. If $\context$ is minimal, then $L = \minbool{a}$ for $a\in L$. Thus $\Ob(\context) =\{\minbool{a}\}$ and $\Mor(\context) =\emptyset$. Hence any functor from $(\context)^{\op}$ is uniquely determined by its action on $\minbool{a}$. Since $(\idelta(b))_{\minbool{a}} = \stone(b)$ $\forall b\in L$, $\idelta(b)$ is a closed-and-open subpresheaf of $\Sigma^{L}$ $\forall b\in L$. If $\context$ is completely disjoint, then $V\not\subseteq W$ $\forall V,W\in \context$. Thus $\Ob(\context) =\{\minbool{a} \mid a\in L\}$ and $\Mor(\context) =\emptyset$. Hence any functor from $(\context)^{\op}$ is uniquely determined by its action on $\Ob(\context)$. Since, by definition, $(\idelta(b))_{\minbool{a}} \in \mathrm{clop}(\sps(\minbool{a}))$ $\forall a,b\in L$, $\idelta(b)$ is a closed-and-open subpresheaf of $\Sigma^{L}$ $\forall b\in L$.
\end{proof}

\section{Introduction}

\subsection{Spectral presheaves over W$^*$-algebras}

The postulate to \cytat{analyse possible connections of the state vector of a quantum system in a Hilbert space with the structures of the new spaces of Grothendieck}, %\footnote{I.e. \cytat{\fontencoding{T2A}\selectfont {(...) можно лишь проанализировать возможные связи вектора состояния квантовой системы в гильбертовом пространстве со структурами новых пространств Гротендика.}}}, 
where a suitable topos would be constructed \cytat{as a category of all sheaves of wave functions} %\footnote{I.e. \cytat{\fontencoding{T2A}\selectfont {(...) понятие топоса как категории всех пучков волновых функции (...)}}.} 
has been proposed in 1974 by Soviet philosopher Igor' A. Akchurin \cite[pp. 186, 192]{Akchurin:1974}\footnote{Akchurin's work was remarkably ahead of its own time, offering  wide rethinking of the mathematical structures of theoretical physics (and their philosophical underpinnings) from the perspective that considered topos (and related cohomological structures) as the foundational notion for the new concept of a space, less than three years after publication of the founding works on topos theory \cite{Artin:Grothendieck:Verdier:1972:1973,Lawvere:1971,Kock:Wraith:1971,Lawvere:1972:book}. Akchurin also proposed to consider order theoretic axiomatisation of special and general relativity (along the lines of \cite{Aleksandrov:Ovchinnikova:1953,Zeeman:1964,Busemann:1967,Kronheimer:Penrose:1967,Pimenov:1968}) as a base category for a topos of general relativity \cite[pp. 176--177]{Akchurin:1974} (the first mathematical construction implementing this idea, which was also the first mathematical paper providing a topos theoretic formulation of general relativity, was proposed by Aleksandr K. Guc \cite{Guc:1991,Guc:Grinkevich:1996}). Our work (together with \cite{Kostecki:2026:causal,Engel:Kostecki:2026:I,Engel:Kostecki:2026:II}) provides a mathematical background for joint implementation of (a particular variant of) both of these ideas, since the complete orthocomplemented lattice $(L,{}^\perp)$ of the spectral presheaf $\Sigma^L$ can be given by the orthomodular lattice of projections on (and thus, also of all rays in) a Hilbert space \cite[III.\S5]{vonNeumann:1932:grundlagen}  \cite[\S\S6--9, 11]{Birkhoff:vonNeumann:1936} \cite[p. 730]{Husimi:1937} (and, more generally, by the lattice of projective units of an order unit space with an archimedean order unit \cite[Thm. 4.5]{Alfsen:Shultz:1976}), by the (generally not orthomodular) lattice of factor W$^*$-subalgebras of a factor W$^*$-algebra \cite[pp. 117--118]{Murray:vonNeumann:1936}, as well as by the (not necessarily orthomodular \cite[p. 6396]{Casini:2002}) lattice of causally closed sets of a lorentzian spacetime \cite{Cegla:Jadczyk:1977,Casini:2002,Casini:2003,Nobili:2006,Cegla:Florek:Jancewicz:2017}. By this reason, the algebra that we introduce (Definition \ref{def.biquasiintuitionistic}.\ref{def.biquasiintuitionistic.2}) to capture the internal structure of a spectral presheaf will be called an \emph{Akchurin algebra}.}. The first mathematical analysis of the space of pure states of a quantum mechanical system through the lenses of a suitable topos (more specifically, an analysis of the soundness and completeness of several formal systems of quantum logic in a category of presheaves over an orthomodular lattice) was made fifteen years later by Vladimir L. Vasyukov \cite{Vasyukov:1989,Vasyukov:2005}.\footnote{The first work discussing topos theoretic semantics for relevance logics (more specifically, $\mathbf{RQ}$) appeared in the same year \cite{Shalak:1989} (see also \cite{Vasyukov:2003} for another approach to this problem).} Another early approach considered the sheaf of germs of convex functions on the space of normalised trace class operators on a separable Hilbert space \cite{Adelman:Corbett:1995:quantum,Adelman:Corbett:1995:sheaf,Corbett:Adelman:2001,Corbett:Durt:2009,Corbett:2012,Corbett:2019}. However, wider development of topos theoretic constructions aimed at analysis of mathematical structure of quantum mechanics started only with the series of papers by Jeremy Butterfield, Andreas D\"{o}ring, and Chris Isham \cite{Isham:Butterfield:1998:I,Butterfield:Isham:1999,Hamilton:Isham:Butterfield:2000,Butterfield:Isham:2002,Doering:Isham:2008:I,Doering:Isham:2008:II,Doering:Isham:2008:III,Doering:Isham:2008:IV,Doering:Isham:2011,Doering:2009,Doering:2011,Doering:Isham:2012,Doering:2016,Doering:2015}\footnote{For the sake of precision, we note that \cite{Hamilton:Isham:Butterfield:2000} was coauthored with John Hamilton (cf. also \cite{Hamilton:2000}).}. This approach aims at reexpressing the contents of the Hilbert space based framework for quantum mechanics by using suitable presheaves inside the topos $\Set^{{(\mathcal{V}(\BH))}^{\mathrm{op}}}$ of contravariant set-valued presheaves over the category $\mathcal{V}(\BH)$ \cite[p. 1415]{Hamilton:Isham:Butterfield:2000}, consisting of objects given by  commutative von Neumann subalgebras of the algebra $\BH$ of bounded linear operators on a Hilbert space $\mathcal{H}$, and morphisms given by the set theoretic inclusions. (More generally \cite{Hamilton:Isham:Butterfield:2000,Doering:2009}, $\BH$ (resp., $\mathcal{V}(\BH)$) can be replaced by an arbitrary W$^*$-algebra $\N$ (resp., the category $\mathcal{V}(\N)$ of commutative W$^*$-subalgebras of $\N$, ordered by an inclusion). All of the statements of this subsection hold under this replacement.) 

The key notion of this approach is the \textit{spectral presheaf} \cite[Def. 2.3]{Hamilton:Isham:Butterfield:2000}, $\Sigma^{\mathcal{H}}\in\Ob(\Set^{{(\mathcal{V}(\BH))}^{\mathrm{op}}})$, with objects given by the Gel'fand spectra $\sp_{\mathrm{G}}(V)$ (i.e. the sets of all nonzero $^*$-algebra homomorphisms from $V$ to $\CC$ \cite[p. 431]{Gelfand:1939}) of the objects $V\in\Ob(\mathcal{V}(\BH))$, and morphisms given by restrictions of these spectra. For any commutative unital C$^*$-algebra $\A$, $\sp_{\mathrm{G}}(\A)$ is equipped with a compact Hausdorff topology \cite[p. 431]{Gelfand:1939} \cite[Prop. 9]{Gelfand:1941}. The set $\Subclop(\Sigma^{\mathcal{H}})$ of all subobjects of $\Sigma^{\mathcal{H}}$ such that $\Sigma^{\mathcal{H}}(V)$ are closed-and-open in this topology $\forall V\in\Ob(\mathcal{V}(\BH))$ \cite[p. 1423]{Hamilton:Isham:Butterfield:2000} has been shown to be a complete Heyting algebra \cite[Thm. 2.5]{Doering:Isham:2008:II}, as well as a complete Brouwer algebra \cite[p. 164]{Doering:2016}.\footnote{In general, the set $\Sub(A)$ of subpresheaves of any presheaf $A$ in any topos of presheaves is a Heyting algebra and a Brouwer algebra \cite[p. 3]{Lawvere:1989} \cite[pp. 280--281]{Lawvere:1991} \cite[Cor. II.3.3]{Zolfaghari:1992} (=\cite[Cor. 3.3]{Reyes:Zolfaghari:1996}). However, the lattice structure of $\Subclop(\Sigma^{\mathcal{H}})$ is not a sublattice of the lattice structure of $\Sub(\Sigma^{\mathcal{H}})$, since its joins and meets are defined via topological closure and topological interior with respect to the Stone space topology (cf. Proposition \ref{thm: Sigma is complete lattice}). Hence, this is a structurally independent (although suggestively similar) result.} $\Subclop(\Sigma^{\mathcal{H}})$ plays the role of the codomain of the \textit{outer daseinisation} functor, $\odelta:\Proj(\BH)\ra\Subclop(\Sigma^{\mathcal{H}})$ \cite[Def. 2.5]{Doering:Isham:2008:II}, which embeds the complete orthomodular lattice $\Proj(\BH):=\{x\in\BH\mid x=x^2=x^*\}$ of linear projection operators on $\H$ into the spectral presheaf by approximating $P\in\Proj(\BH)$ from above by the elements of each $V\in\Ob(\mathcal{V}(\BH))$, \cytat{liberating the projection operators from the shackles of quantum logic and thursting them down to existential world of Heyting algebras} \cite[p. 12]{Doering:Isham:2008:II}. Since $\odelta$ is an injective, top-, bottom-, join-preserving, order monotone map between two complete lattices, by the adjoint functor theorem for posets \cite[Ex. \S3.J]{Freyd:1964} there is a surjective, top-, botom-, meet-preserving, order monotone map $\oepsilon:\Subclop(\Sigma^{\mathcal{H}})\ra\Proj(\BH)$ such that $\odelta(x)\leq y$ $\iff$ $x\leq\oepsilon(y)$ $\forall x\in\Proj(\BH)$ $\forall y\in\Subclop(\Sigma^{\mathcal{H}})$ \cite[\S7.4]{Cannon:2013} (=\cite[\S4.5]{Cannon:Doering:2018}), i.e. there is an adjunction $\odelta\dashv\oepsilon$. 

Using this adjunction together with an orthocomplementation $^\perp:\Proj(\BH)\ra\Proj(\BH)$, \cite[Def. 4.1]{Eva:2015} introduced a unary operation $(\cdot)^\eva:=\odelta\circ{}^\perp\circ\oepsilon:\Subclop(\Sigma^{\mathcal{H}})\ra\Subclop(\Sigma^{\mathcal{H}})$. It turns out that $^\eva$ is not only paraconsistent (i.e. $x\land x^\eva\geq0$ $\forall x\in\Subclop(\Sigma^{\mathcal{H}})$) \cite[Thm. 4.2.(ii)]{Eva:2015}, but it is also fully paraconsistent (i.e. $x\land x^\eva=0$ if{}f $x\in\{\varnothing,\Sigma^{\mathcal{H}}\}$) \cite[Thm. 5.13]{Doering:Eva:Ozawa:2021}. Hence, additionally to Heyting and Brouwer negations, $\Subclop(\Sigma^{\mathcal{H}})$ is equipped with a third negation operator, that represents the orthocomplementation of $\Proj(\BH)$ and features a strong variant of paraconsistency. Since the varieties of Heyting and Brouwer algebras provide the sound and complete\footnote{In a somewhat unfortunate clash of terminology, the completeness of a logic with respect to a class of lattices is a priori unrelated to the completeness of these lattices, and both of them are a priori unrelated to paracompleteness of negation operator. All of these notions are featured in this paper without exhibiting an ambiguity of the intended designate.} algebraic semantics for, respectively, intuitionistic logic $\Int$ and cointuitionistic logic ${\coInt}$ (cf. Remark \ref{rem.logics}.\ref{rem.logics.moisil.rauszer}), it is quite natural to ask: what is the logic and the corresponding algebraic semantics such that $(\Subclop(\Sigma^{\mathcal{H}}),{}^{\eva})$ is a sound model of it? \cite[Thm. 5.1]{Eva:2015} and \cite[Cor. 2.1.9]{Eva:2016} contain a claim, not backed up by a proof, that, given a definition of the implication connective by means of $x\limp y:=x^\eva\lor y$ $\forall x,y\in\Subclop(\Sigma^{\mathcal{H}})$, this logic is specifically the relevance logic $\mathbf{DKQ}$\footnote{In both these papers it is called $\mathbf{DLQ}$, however the axioms in \cite[pp. 166--167]{Eva:2015} are the axioms of $\mathbf{DKQ}$ (cf., e.g., \cite[pp. 355--358]{Brady:1984}).}.

\subsection{Spectral presheaves over orthomodular lattices}

The Stone spectrum $\sps(B)$ of a boolean lattice $B$, defined as the set of all boolean lattice homomorphisms from $B$ to $\{0,1\}$ \cite[Thm. III]{Stone:1934}, is an analogue of the Gel'fand spectrum for boolean lattices. This set can be equipped with a topology that turns it into a Stone space (i.e. a totally disconnected, compact, Hausdorff topological space), allowing to reconstruct $B$ as the lattice of closed-and-open sets of $\sps(B)$ \cite[Thms. IV$_1$, IV$_2$]{Stone:1934} \cite[Thms. 1, 3, 4]{Stone:1937}. While the Gel'fand spectrum (and the corresponding duality, allowing to reconstruct $\A$ as the set of all continuous functions from $\sp_{\mathrm{G}}$ to $\CC$ \cite[Thm. 6]{Gelfand:1939} \cite[Thm. 1]{Gelfand:Naimark:1943}) is applicable to any unital commutative C$^*$-algebra $\A$ \cite[p. 431]{Gelfand:1939} \cite[\S\S5, 7]{Gelfand:1941}, the class of unital commutative C$^*$-algebras such that their Gel'fand spectra are Stone spaces (so they correspond to the Stone spectra of arbitrary boolean algebras) is strictly smaller, and is given by the unital commutative \cb{S}tr\u{a}til\u{a}--Zsid\'{o} C$^*$-algebras \cite[\S\S9.1, 9.7]{Stratila:Zsido:1977}. Hence, the Gel'fand spectrum, when considered on its own, has wider applicability than the Stone spectrum. However, $\odelta$ is defined using meets over not necessarily finite sets (cf. \eqref{eqn.outer.daseinisation}), requiring to restrict considerations from arbitrary boolean subalgebras to those that are complete. Since the category of complete boolean algebras (resp., complete boolean algebras allowing semifinite strictly positive countably additive measure) with order continuous boolean homomorphisms as morphisms is equivalent to the category of unital commutative AW$^*$-algebras (resp., W$^*$-algebras) with unital complete (resp., normal) $^*$-homomorphisms as morphisms (cf. \cite[\S5.8]{Kostecki:2013} for a detailed discussion with further references), the setting of unital AW$^*$-algebras provides an upper bound of generality for the spectral presheaf approach based on the Gel'fand spectrum.

This can be seen as a point of departure of \cite{Cannon:2013,Cannon:Doering:2018}, which consider an arbitrary orthomodular (resp., complete orthomodular) lattice $L$, replacing the category $\mathcal{V}(\BH)$ (and, more generally, $\mathcal{V}(\mathcal{A})$, for a unital AW$^*$-algebra $\mathcal{A}$) with the category $\contextnotcomplete$ (resp., $\context$) of boolean (resp., complete boolean) sublattices of $L$ as objects and set theoretic embeddings as morphisms, resulting in the topos $\Set^{{(\contextnotcomplete)}^{\mathrm{op}}}$ (resp., $\Set^{{(\context)}^{\mathrm{op}}}$).\footnote{A more general programme of a theory of spectra, observables, and corresponding presheaves for any lattice was proposed and developed by de Groote \cite{deGroote:2001,deGroote:2005,deGroote:2005:I,deGroote:2007,deGroote:2011}. In particular, the main idea underlying outer (resp., inner) daseinisation, i.e. using $\bigwedge\{y\in B\mid y\geq x\}$ in \eqref{eqn.outer.daseinisation} (resp., $\bigvee\{y\in B\mid y\leq x\}$ in \eqref{eqn.inner.daseinisation}), has been originally proposed in \cite[Def. 3.2]{deGroote:2005}.} (In comparison, \cite{Vasyukov:1989,Vasyukov:2005} studies the topos $\Set^L$.) The major change amounts to defining $\Sigma^L$ as an assignment of Stone spectra to the objects of $\contextnotcomplete$ (resp., $\context$) and restrictions of these spectra to morphisms. This allows for consideration of a wide class of models beyond unital AW$^*$-algebras (e.g., $L$ can be given by a complete orthomodular lattice of idempotents of a nonassociative JBW-algebra \cite[Thm. 4]{Topping:1965} \cite[Thm. III.\S6.1]{Sarymsakov:Ayupov:Hadzhiev:Chilin:1983}). The set $\Subclop(\Sigma^L)$ and the adjunction $\odelta\dashv\oepsilon$ are constructed for $\context$ analogously to the former case of $\mathcal{V}(\BH)$. Furthermore, for $\Sigma^L\in\Ob(\Set^{{(\context)}^{\mathrm{op}}})$ there is a complete orthomodular lattice isomorphism between $L$ and $(\Subclop(\Sigma^L))/\approx_\oepsilon$, where $\approx_\oepsilon$ is an equivalence relation defined by $x\approx_\oepsilon y$ if{}f $\oepsilon(x)=\oepsilon(y)$ $\forall x,y\in\Subclop(\Sigma^L)$ \cite[Thm. 7.5.4]{Cannon:2013} \cite[p. 166]{Eva:2015} \cite[Thm. 4.19]{Cannon:Doering:2018}. Additionally, any two orthomodular lattices are isomorphic (in the category of orthomodular lattices and orthomodular lattice homomorphisms) if{}f their corresponding spectral presheaves are isomorphic (in the category of Stone space-valued presheaves with morphisms given by pairs of a homomorphism of underlying orthomodular lattice and a natural transformation between the corresponding presheaves) \cite[Thm. 5.7.3]{Cannon:2013} (=\cite[Thm. 3.18]{Cannon:Doering:2018}).  Hence, one can consider the study of the structure of $\Sigma^L$ and $\Subclop(\Sigma^L)$ as an extension of the study of the structure of orthomodular lattices $L$ on their own.

\subsection{New results}

The new results contained in the current paper are:
\begin{enumerate}[nosep,label=(\roman*)]
\item\label{item.new.results.i} As a prequel to \cite{Engel:Kostecki:2026:I}, in Section \ref{section.vakarelov.akchurin} we introduce \emph{quasiintuitionistic}, \emph{coquasiintuitionistic}, \emph{biquasiintuitionistic}, and \emph{Akchurin algebras}. They are defined within the general framework of Vakarelov's algebraic approach to logic \cite{Vakarelov:1976,Vakarelov:1989,Orlowska:Vakarelov:2005}, which provides equivalent (i.e. sound and complete) algebraic models of lattice (or poset) logics with a negation (or a finite number of negations).\footnote{It also provides equivalent relational (i.e. frame semantic) models for these logics. Hence, e.g., Corollary \ref{cor.complete.bqi.akchurin}.\ref{cor.complete.bqi.akchurin.iii} implies that the Akchurin logic of the spectral presheaf can be represented in terms of a particular frame semantic model.} We introduce the corresponding logics in Section \ref{section.quasiintuitionistic.Akchurin.logics}. In particular, the class of all biquasiintuitionistic algebras is an equivalent model of the \emph{biquasiintuitionistic logic} $\biQInt$ (Proposition \ref{prop.soundness.of.coQInt}.\ref{prop.soundness.of.coQInt 3}), defined as a logic with two negations that exhibit all of the properties of the respective two negations of the biintuitionistic logic $\biInt$ except of the relationships with the implication and coimplication, since the latter are absent in $\biQInt$. E.g., one of the negations in $\biQInt$ is paraconsistent, while the other is paracomplete. Furthermore, in Corollary \ref{cor.akchurin.is.sound.and.complete} we prove the class of all Akchurin algebras is an equivalent model of the \emph{Akchurin logic} $\biQInt\otimes\biInt$, defined as the product of biquasiintuitionistic and biintuitionistic logic over the common set of formulas. In this paper we will use (in a selfcontained way) only some part of the algebraic structure of the above framework, which we develop more extensively in \cite{Engel:Kostecki:2026:I}.
\item\label{item.new.results.ii} In Section \ref{section.spectral.presheaf.daseinisations} we extend the construction of the spectral presheaf from complete orthomodular to complete orthocomplemented lattices, together with the corresponding outer and inner daseinisations, $\odelta$ and $\idelta$ (the latter was introduced for Gel'fand spectra in \cite[Def. 2.6]{Doering:Isham:2008:II}, but was not considered before in lattice theoretic generality), and adjunctions $\odelta\dashv\oepsilon$ and $\iepsilon\dashv\idelta$. This generalisation allows to study, in particular, the spectral presheaves over complete orthocomplemented lattices of:
\begin{enumerate}[nosep,label=\alph*)]
\item factor W$^*$-subalgebras of a factor W$^*$-algebra, with the partial order given by set inclusion, the meet given by set intersection, the join given by double commutant of set sum, and the orthocomplementation given by commutant \cite[pp. 117--118]{Murray:vonNeumann:1936};
\item causally closed subsets of a lorentzian spacetime $(M,g)$, arising as image of the double causal complementation monad, with the complementation of a subset $Y$ of $M$ defined as the set of all points of $M$ that are not causally connected with $Y$ (the definition of `causal connection' and the range of admitted spacetimes and subsets varies; while some of these lattices are orthomodular \cite[\S3]{Cegla:Jadczyk:1977} \cite[\S4]{Casini:2002} \cite[App. 2]{Nobili:2006} \cite[\S2]{Cegla:Florek:Jancewicz:2017}, it is not so in general \cite[p. 6396]{Casini:2002} \cite[\S\S2.2, 3]{Casini:2003}).
\end{enumerate}
The above two examples are discussed in more details in \cite{Kostecki:2026:causal}.\footnote{The idea to apply spectral presheaves to study complete orthomodular lattices of causally closed subsets was first expressed by de Groote \cite[p. 5]{deGroote:2007} \cite[p. 268]{deGroote:2011}. Spectral presheaves over arbitrary complete orthocomplemented lattices and use of $(\cdot)^\eva:=\odelta\circ{}^\perp\circ\oepsilon(\cdot)$ as a paraconsistent internalisation of the causal complement and of the factor W$^*$-algebra commutant were first considered in \cite{Kostecki:2024:Waterloo}.}
\item\label{item.new.results.iii} In Section \ref{section.akchurin.spectral.presheaf} we introduce two negation operators on $\Subclop(\Sigma^L)$, $(\cdot)^\eva:=\odelta\circ{}^\perp\circ\oepsilon(\cdot)$ and $(\cdot)^\coeva:=\idelta\circ{}^\perp\circ\iepsilon(\cdot)$. The former is a direct generalisation of $^\eva$ introduced in \cite[Def. 4.1]{Eva:2015} and is paraconsistent. The latter is paracomplete, and has not been considered before. The proof of \cite[Prop. 7.2.6]{Cannon:2013} (=\cite[Prop. 4.9]{Cannon:Doering:2018}), stating that $(\Subclop(\Sigma^L),\subseteq,\land,\lor,\varnothing,\Sigma^L)$ carries a complete Skolem algebra structure, does not rely on the assumption of orthomodularity of $(L,{}^\perp)$, so $(\Subclop(\Sigma^L),\subseteq,\land,\lor,\varnothing,\Sigma^L,\himp,\hnot,\cohimp,\cohnot)$ is a sound model of $\biInt$ for an arbitrary complete orthocomplemented lattice $(L,{}^\perp)$. In Corollaries \ref{cor.coeva.qint}.\ref{cor.coeva.qint.2} and \ref{cor.complete.bqi.akchurin}.\ref{cor.complete.bqi.akchurin.iii} we show that $(\Subclop(\Sigma^L),\subseteq,\land,\lor,\varnothing,\Sigma,{}^\coeva,{}^\eva)$ is a complete biquasiintuitionistic algebra (hence, a sound model of $\biQInt$) while $(\Subclop(\Sigma^L),\subseteq,\land,\lor,\himp,\hnot,\cohimp,\cohnot,\varnothing,\Sigma,{}^\coeva,{}^\eva)$ is a complete Akchurin algebra (hence, a sound model of $\biQInt\otimes\biInt$).
\item\label{item.new.results.iv}\sloppy For any quasiintuitionistic (resp., coquasiintuitionistic) algebra $(K,\leq,\land,\lor,0,1,{}^\coeva)$ (resp., $(K,\leq,\land,\lor,0,1,{}^\eva)$), $(\cdot)^{\coeva\coeva}$ is an interior (resp., $(\cdot)^{\eva\eva}$ is a closure) operator, i.e. it is a comonad (resp., monad). Hence, by \cite[p. 16]{Freyd:1972}, it defines a full coreflective (resp., full reflective) subcategory $(K_{\coeva\coeva}:=\{x\in K\mid x=x^{\coeva\coeva}\},\leq)$ (resp., $(K_{\eva\eva}:=\{x\in K\mid x=x^{\eva\eva}\},\leq)$). In general, $(\cdot)^{\coeva\coeva}$ (resp., $(\cdot)^{\eva\eva}$) does not preserve finite meets (resp., finite joins). On the other hand, $(\cdot)^{\coeva\coeva}$ (resp., $(\cdot)^{\eva\eva}$) preserves finite joins (resp., finite meets), with the join (resp., meet) operation on $(K_{\coeva\coeva},\leq)$ (resp., $(K_{\eva\eva},\leq)$) given by $\cdot\curlyvee\cdot:=(\cdot\lor\cdot)^{\coeva\coeva}$ (resp., $\cdot\curlywedge\cdot:=(\cdot\land\cdot)^{\eva\eva}$).\footnote{In case of $(K,\leq,\land,\lor,0,1,{}^\eva)$ given by a Brouwer algebra, this construction goes back to \cite[Def. 4.1, Thm. 4.2]{McKinsey:Tarski:1946}.} In both cases $0$ and $1$ are preserved. This way, the image of the action of $(\cdot)^{\coeva\coeva}$ (resp., $(\cdot)^{\eva\eva}$) on a quasiintuitionistic (resp., coquasiintuitionistic) algebra $(K,\leq,\land,\lor,0,1,{}^{\coeva})$ (resp.,  $(K,\leq,\land,\lor,0,1,{}^{\eva})$) determines an \emph{internal algebra} $(K_{\coeva\coeva},\leq,\land,\curlyvee,0,1,{}^\coeva)$ (resp., $(K_{\eva\eva},\leq,\curlywedge,\lor,0,1,{}^\eva)$). We show that both of these algebras are orthocomplemented lattices. This generalises \cite[Thm. 4.5]{McKinsey:Tarski:1946}, stating that the internal algebra of a Brouwer algebra is a boolean algebra. An analogous result for Heyting algebras, following via the duality between meets and joins, is an algebraic equivalent of the \emph{Kolmogorov--Glivenko theorem} \cite[III--IV]{Kolmogorov:1925} \cite[p. 183]{Glivenko:1929}, which identifies the double negated formulas in the intuitionistic logic with the formulas of the classical logic (the dual statement for the cointuitionistic logic, equivalent to \cite[Thm. 4.5]{McKinsey:Tarski:1946}, is known as the \emph{dual Kolmogorov--Glivenko theorem} \cite[Thm. (p. 473)]{Czermak:1977} \cite[Thm. 1]{Goodman:1981} \cite[Thm. 2.1]{Urbas:1996}). So, our result provides a generalised (resp., generalised dual) Kolmogorov--Glivenko theorem, relating $\QInt$ (resp., $\coQInt$) with its internal orthocomplemented lattice logic (this logic has first appeared in \cite[Def. 1.1]{Goldblatt:1974}) at the expense of restriction to formulas that do not contain meets (resp., joins). Since for any Skolem algebra its two internal boolean lattices are isomorphic, an Akchurin algebra admits three, a priori different, internal lattices: two orthocomplemented and one boolean. These results are a subject of Section \ref{section.internal.lattice.akchurin}.
\item\label{item.new.results.vv} In Section \ref{section.internal.lattice.subclop} we show that the two internal orthocomplemented lattices of the Akchurin algebra of a spectral presheaf $\Sigma^L$ are mutually isomorphic, and they are also isomorphic to $(L,{}^\perp)$. This way we prove that the orthocomplemented lattice underlying the spectral presheaf can be uniquely reconstructed from the internal Akchurin logic of this presheaf as a full coreflective (resp., full reflective) subcategory of the double quasiintuitionistic (resp., coquasiintuitionistic) negation comonad (resp., monad), with a natural logical interpretation of this construction as an instance of the generalised (resp., generalised dual) Kolmorogov--Glivenko theorem. Proposition \ref{prop: Precqi internal oml} (resp., Corollary \ref{cor.quasiintuitionistic.internal.orthomodular.lattice}) characterises those coquasiintuitionistic (resp., quasiintuitionistic) algebras whose internal orthocomplemented lattices are orthomodular. This leads to Corollary \ref{cor.orthomodularity.from.subclop}, characterising orthomodularity of a complete orthocomplemented lattice $(L,{}^\perp)$ in terms of the conditions on the coquasiintuitionistic (or, equivalently, quasiintuitionistic) structure of $\Subclop(\Sigma^L)$.
\item\label{item.new.results.v} In Corollary \ref{cor.nogo.for.relevance.logics.in.Subclop} we prove that the claim of  \cite[Thm. 5.1]{Eva:2015} and \cite[Cor. 2.1.9]{Eva:2016}, stating that $(\Subclop(\Sigma^L),\subseteq,\land,\lor,\varnothing,\Sigma^L,{}^\eva,\Rightarrow)$ is a model of a paraconsistent logic $\mathbf{DKQ}$, is strictly false. We also exclude all other relevance logics (as long as these exhibit a De Morgan algebra as an extensional reduct of their equivalent algebraic semantics, which is a standard assumption for this class of logics).
\end{enumerate}
One of the open problems that arise from these results is to characterise the Akchurin logic of spectral presheaves among all Akchurin logics. This requires to understand better how much of the structure of the spectral presheaf can be characterised by means of the properties of its Akchurin algebra. Further discussion of the above issues, together with a further generalisation of the spectral presheaf framework, will be carried out in \cite{Engel:Kostecki:2026:I,Engel:Kostecki:2026:II}. 

\section{\label{section.vakarelov.akchurin}Vakarelov, (co/bi/)quasiintuitionistic, and Akchurin algebras}

\begin{definition}
For any lattice $(K,\leq,\land,\lor)$, a map $\Theta:K\ra K$ such that $\Theta(x\land y)=x\lor y$ and $\Theta(x\lor y)=x\land y$ $\forall x,y\in K$ will be called a \emph{$(\land,\lor)$-duality}.
\end{definition}

\begin{definition}
A lattice $(K,\leq,\land,\lor)$ is called:
\begin{enumerate}[nosep, label=(\roman*)]
\item \emph{complete} if{f} every subset $M\subseteq K$ has a greatest lower bound and a least upper bound in $K$;
\item \emph{bounded} if{}f $\exists0,1\in K$ such that $0\leq x$ and $y\leq 1$ $\forall x,y\in K$;
\item \emph{distributive} if{}f any of the following equivalent conditions holds $\forall x,y,z\in K$:
\begin{enumerate}[nosep, label=\alph*)]
\item $x\land(y\lor z)=(x\land y)\lor(x\land z)$;
\item $x\lor(y\land z)=(x\lor y)\land(x\lor z)$.
\end{enumerate}
\end{enumerate}
\end{definition}

\begin{definition}\label{def.negation}
Let $(K,\leq)$ be a poset, equipped with the map $^\neg:K\ra K$. Consider the following conditions $\forall x,y\in K$:
\begin{enumerate}[nosep,label=\textup{n\arabic*)}]
\begin{minipage}{0.5\linewidth}
\item if $x\leq y$, then $y^\neg\leq x^\neg$;\label{def: negation 1}
\item $x\leq x^\neg{}^\neg$;\label{def: negation 2}
\end{minipage}
\begin{minipage}{0.5\linewidth}
\item $x^\neg{}^\neg\leq x$.\label{def: negation 3}
\item $x^\neg=x^{\neg\neg\neg}$.\label{def: negation 8}
\end{minipage}
\end{enumerate}
If $(K,\leq,\land,\lor)$ is a lattice, then consider also the conditions $\forall x,y\in K$:
\vspace{2pt}\begin{enumerate}[nosep,label=\textup{n\arabic*)}]
\setcounter{enumi}{4}
\begin{minipage}{0.5\linewidth}
\item if $x\land y\leq z$, then $x\land z^\neg\leq y^\neg$;\label{def: negation 4}
\item if $x\leq y\lor z$, then $z^\neg\leq x^\neg\lor y$;\label{def: negation 5}
\item $x\land x^\neg\leq y$;\label{def: negation 6}
\item $y\leq x\lor x^\neg$;\label{def: negation 7}
\end{minipage}
\begin{minipage}{0.5\linewidth}
\item $x^\neg\land y^\neg\leq(x\lor y)^\neg$;\label{def: negation 9}
\item $x^\neg\land y^\neg\geq(x\lor y)^\neg$;\label{def: negation 10}
\item $x^\neg\lor y^\neg\leq(x\land y)^\neg$;\label{def: negation 11}
\item $x^\neg\lor y^\neg\geq(x\land y)^\neg$.\label{def: negation 12}
\end{minipage}
\end{enumerate}
\vspace{3pt}If $(K,\leq,\land,\lor,0,1)$ is a bounded lattice, then consider also the conditions $\forall x\in K$:
\vspace{2pt}
\begin{enumerate}[nosep,label=\textup{n\arabic*)}]
\setcounter{enumi}{12}
\begin{minipage}{0.5\linewidth}
\item $x\land x^\neg=0$;\label{def: negation 13}
\item $x\lor x^\neg=1$;\label{def: negation 14}
\item $x\land x^\neg\geq0$;\label{def: negation 15}
\end{minipage}
\begin{minipage}{0.5\linewidth}
\item $x\lor x^\neg\leq1$;\label{def: negation 16}
\item $1^\neg=0$;\label{def: negation 17}
\item $0^\neg=1$.\label{def: negation 18}
\end{minipage}
\end{enumerate}
\vspace{2pt}
$^\neg$ will be called a \emph{$p$ negation} if{}f it satisfies $q$, where the pairs $p:q$ are given by:\footnote{A subminimal negation is sometimes called a \emph{contraposition}.}
\vspace{3pt}
\begin{enumerate}[nosep,label=(\roman*)]
\begin{minipage}{0.5\linewidth}
\item\label{def.negation.1} \emph{subminimal}: \ref{def: negation 1};  
\item\label{def.negation.10} \emph{quasiintutionistic}: \ref{def: negation 1}, \ref{def: negation 2}, and \ref{def: negation 6}; 
\item\label{def.negation.11} \emph{coquasiintuitionistic}: \ref{def: negation 1}, \ref{def: negation 3}, and \ref{def: negation 7};
\item\label{def.negation.12} \emph{intuitionistic}: \ref{def: negation 1}, \ref{def: negation 2}, \ref{def: negation 4}, and \ref{def: negation 6}; 
\end{minipage}
\begin{minipage}{0.5\linewidth}
\item\label{def.negation.13} \emph{cointuitionistic}: \ref{def: negation 1}, \ref{def: negation 3}, \ref{def: negation 5}, and \ref{def: negation 7};
\item\label{def.negation.16} \emph{paraconsistent}: \ref{def: negation 15};
\item\label{def.negation.19} \emph{paracomplete}: \ref{def: negation 16}.
\end{minipage}
\end{enumerate}
\vspace{3pt}
$^\neg$ will be called:
\begin{enumerate}[nosep,label=(\roman*)]
\setcounter{enumi}{7}
\item\label{def.negation.9} an \emph{involution} if{}f \textup{\cite[p. 8]{Birkhoff:1940}} it satisfies \ref{def: negation 1}, \ref{def: negation 2}, and \ref{def: negation 3};
\item\label{def.negation.22} an \emph{orthocomplementation} if{}f \textup{\cite[p. 830]{Birkhoff:vonNeumann:1936}} it satisfies \ref{def: negation 1}, \ref{def: negation 2}, \ref{def: negation 3}, \ref{def: negation 6}, and \ref{def: negation 7}.
\end{enumerate}
\end{definition}

\begin{lemma}\label{lemm.bullet.negation.properties}
Let $(K,\leq,\land,\lor)$ be a lattice, equipped with the map $^\neg:K\ra K$. 
\begin{enumerate}[nosep, label=(\roman*)]
\item If ${}^\neg$ satisfies \ref{def: negation 2} (resp., \ref{def: negation 3}), then it satisfies \ref{def: negation 9} (resp. \ref{def: negation 12}).\label{lemm.bullet.negation.properties 1}
\item ${}^\neg$ satisfies \ref{def: negation 1} if{f} it satisfies \ref{def: negation 10} if{f} it satisfies \ref{def: negation 11}.\label{lemm.bullet.negation.properties 2}
\item If ${}^\neg$ satisfies \ref{def: negation 1} and (\ref{def: negation 2} or \ref{def: negation 3}) then it satisfies \ref{def: negation 8}.\label{lemm.bullet.negation.properties 3}
\item If ${}^\neg$ satisfies \ref{def: negation 1} and \ref{def: negation 2} (resp., \ref{def: negation 1} and \ref{def: negation 3}), then it satisfies \ref{def: negation 8}, \ref{def: negation 9}, \ref{def: negation 10}, \ref{def: negation 11} (resp., \ref{def: negation 8}, \ref{def: negation 10}, \ref{def: negation 11}, \ref{def: negation 12}).\label{lemm.bullet.negation.properties 4}
\item If $(K,\leq,\land,\lor)$ is bounded, then \ref{def: negation 6} (resp., \ref{def: negation 7}) is equivalent with \ref{def: negation 13} (resp., \ref{def: negation 14}).\label{lemm.bullet.negation.properties 5}
\item If $(K,\leq,\land,\lor)$ is bounded and ${}^\neg$ satisfies \ref{def: negation 1} and \ref{def: negation 14} (resp., \ref{def: negation 3} and \ref{def: negation 18}; \ref{def: negation 1} and \ref{def: negation 13}; \ref{def: negation 2} and \ref{def: negation 18}), then it satisfies \ref{def: negation 18} (resp., \ref{def: negation 17}; \ref{def: negation 17}; \ref{def: negation 18}).\label{lemm.bullet.negation.properties 45}
\item\label{lemm.bullet.negation.properties 7} If $(K,\leq,\land,\lor)$ is bounded and ${}^{\neg}$ satisfies \ref{def: negation 1} and \ref{def: negation 2} (resp., \ref{def: negation 1} and \ref{def: negation 3}), then it satisfies \ref{def: negation 18} (resp., \ref{def: negation 17}).
\item\label{lemm.bullet.negation.properties y} If ${}^{\neg}$ satisfies \ref{def: negation 1}, \ref{def: negation 2}, and \ref{def: negation 3}, then ${}^{\neg}$ satisfies \ref{def: negation 6} if{f} it satisfies \ref{def: negation 7}.
\end{enumerate}
\end{lemma}

\begin{proof}
\begin{enumerate}[nosep, label=(\roman*)]
\item \cite[Lem. 4.12]{Dunn:Zhou:2005} (resp., \cite[Lem. 2.14]{Dunn:Zhou:2005}).
\item Let $x,y,a,b\in K$. Assume that ${}^{\neg}$ satisfies \ref{def: negation 1}. By applying \ref{def: negation 1} to $x\leq x\lor y\geq y$, we obtain $x^\neg\geq(x\lor y)^\neg\leq y^\neg$. However, by definition of $\land$, $a\land b$ is the largest element such that $a\land b\leq a$ and $a\land b\leq b$. Hence, $(x\lor y)^\neg\leq x^\neg\land y^\neg$, proving \ref{def: negation 10}. Assume that ${}^{\neg}$ satisfies \ref{def: negation 10} and $x\leq y$. Then $y^{\neg} = (x\lor y)^{\neg} \leq x^{\neg} \land y^{\neg} \leq x^{\neg}$, so ${}^{\neg}$ satisfies \ref{def: negation 1}, showing that ${}^{\neg}$ satisfies \ref{def: negation 1} if{f} it satisfies \ref{def: negation 10}. The proof of \ref{def: negation 1} if{f} \ref{def: negation 11} follows by $(\land,\lor)$-duality. So ${}^{\neg}$ satisfies \ref{def: negation 10} if{f} it satisfies \ref{def: negation 1} if{f} it satisfies \ref{def: negation 11}. 
\item Let $x\in K$. \ref{def: negation 3} implies $(x^\neg)^{\neg\neg}\leq x^\neg$, while \ref{def: negation 1} and \ref{def: negation 3} imply $x^{\neg\neg\neg}\geq x^\neg$. Hence,  $x^{\neg\neg\neg}=x^\neg$. The proof of \ref{def: negation 8} based on \ref{def: negation 1} and \ref{def: negation 2} follows by $(\land,\lor)$-duality.
\item Follows from \ref{lemm.bullet.negation.properties 1}, \ref{lemm.bullet.negation.properties 2}, and \ref{lemm.bullet.negation.properties 3}.
\item Follows directly from the definition of $\land$ and $0$ (resp., $\lor$ and $1$).
\item \ref{def: negation 1} and \ref{def: negation 14} imply $1=0\lor 0^\neg=0^\neg$. \ref{def: negation 3} and \ref{def: negation 18} imply $1^\neg=0^{\neg\neg}\leq0$, so $1^\neg=0$. The rest of the proof follows by $(\land,\lor)$-duality.
\item If ${}^{\neg}$ satisfies \ref{def: negation 1} and \ref{def: negation 2}, then $0 \leq 1^{\neg}$ implies $1\leq 1^{\neg\neg} \leq 0^{\neg}$, so $0^{\neg} = 1$. If ${}^{\neg}$ satisfies \ref{def: negation 1} and \ref{def: negation 3}, then $0^{\neg} \leq 1$ implies $1^{\neg} \leq 0^{\neg\neg} \leq 0$, so $1^{\neg} = 0$.
\item Assume that ${}^{\neg}$ satisfies \ref{def: negation 1}, \ref{def: negation 2}, \ref{def: negation 3}, and \ref{def: negation 7}. Since \ref{def:  negation 9} and \ref{def: negation 10} are satisfied by \ref{lemm.bullet.negation.properties 4}, $x^{\neg} \land x = x^{\neg} \land x^{\neg\neg} = (x\lor x^{\neg})^{\neg} \leq y^{\neg}$ $\forall x,y\in L$. $L^{\neg} = L$ since $z= z^{\neg\neg} = (z^{\neg})^{\neg}$ $\forall z\in L$, so $x^{\neg} \land x \leq z$ $\forall z \in L^{\neg} = L$, hence ${}^{\neg}$ satisfies \ref{def: negation 6}. The proof of \ref{def: negation 7} follows by $(\land,\lor)$-duality.
\end{enumerate}
\end{proof}

\begin{lemma}\label{lemm.coheyting.negation}
In any bounded lattice $(K,\leq,\land,\lor,0,1)$ with a map $^\neg:K\to K$:
\begin{enumerate}[nosep, label=(\roman*)]
\item if $1^\neg=0$, then equivalent are:
\label{item: lemm.coheyting.negation.coheyting}
\begin{enumerate}[nosep, label=\alph*)]
\item $x\lor y=1$ if{}f $x^\neg\leq y$ $\forall x,y\in K$;\label{item: lemm.coheyting.negation 1}
\item if $x\leq y\lor z$, then $z^\neg\leq x^\neg\lor y$ $\forall x,y,z\in K$;\label{item: lemm.coheyting.negation 2}
\end{enumerate}
\item if $0^\neg=1$, then equivalent are:
\label{item: lemm.coheyting.negation.heyting}
\begin{enumerate}[nosep, label=\alph*)]
\item $x\land y=0$ if{}f $x\leq y^\neg$ $\forall x,y\in K$;\label{item: lemm.coheyting.negation 3}
\item if $x\land y\leq z$, then $x\land z^\neg\leq y^\neg$ $\forall x,y,z\in K$.\label{item: lemm.coheyting.negation 4}
\end{enumerate}
\end{enumerate}
\end{lemma}

\begin{proof}
\begin{enumerate}[nosep, label=(\roman*)]
\item Let $x,y,z\in K$.
\begin{enumerate}[nosep]
\item[a$\ra$b)] If $x\leq y\lor z$, then $y\lor z\lor x^\neg\geq x\lor x^\neg=1$, which is equivalent with $y\lor z\lor x^\neg=1$. Assuming \ref{item: lemm.coheyting.negation 1}, this gives $z^\neg\leq x^\neg\lor y$. 
\item[b$\ra$a)] If $z^\neg\leq y$, then $y\lor z\geq z^\neg\lor z=1$, which is equivalent with $y\lor z=1$. On the other hand, by \ref{item: lemm.coheyting.negation 2}, if $1\leq y\lor z$, then $z^\neg\leq 1^\neg\lor y=0\lor y=y$, hence $y\lor z=1$ implies $z^\neg\leq y$. Hence, $y\lor z=1$ if{}f $z^\neg\leq y$.
\end{enumerate}
\item Follows from \ref{item: lemm.coheyting.negation.coheyting} by $(\land,\lor)$-duality.
\end{enumerate}
\end{proof}

\begin{definition}\label{def.subminimal.DeMorgan}
A lattice $(K,\leq,\land,\lor)$ equipped with $^\neg:K\ra K$ will be called:
\begin{enumerate}[nosep, label=(\roman*)]
\item \textup{\cite[p. 57]{Vakarelov:1976}} a \emph{Vakarelov algebra} if{}f ${}^\neg$ is subminimal;\label{def.subminimal.DeMorgan.subminimal}
\item \textup{\cite[\S52]{Moisil:1936}} \textup{\cite[p. 259]{BialynickiBirula:Rasiowa:1957}} \textup{\cite[p. 485]{Kalman:1958}}\footnote{In \cite[p. 259]{BialynickiBirula:Rasiowa:1957} there is an additional assumption that $(K,\leq,\land,\lor)$ is bounded. On the other hand, \cite[\S52]{Moisil:1936} assumes both De Morgan laws and $x^{\neg\neg}= x$ $\forall x\in K$, but does not assume $x\lor(x\land y)=x\land(x\lor y)=x$ $\forall x,y\in K$ (so, it admits nondistributive lattices).} a \emph{De Morgan algebra} if{}f ${}^\neg$ is an involution and $(K,\leq,\land,\lor)$ is distributive;\label{def.subminimal.DeMorgan.DeMorgan}
\end{enumerate}
\end{definition}

\begin{definition}\label{def.heyting.brouwer.skolem.algebra}
A bounded lattice $(K,\leq,\land,\lor,0,1)$ will be called:\footnote{%
Bounded lattices $(K,\leq,\land,\lor,0,1)$ satisfying both \eqref{eqn.adjunction.heyting.implication} and \eqref{eqn.adjunction.coheyting.implication} were first introduced by Skolem in \cite[Thm. 7]{Skolem:1919}, and were called `Skolem algebras' in \cite[p. I-7]{Esakia:1975}. Skolem (resp., Brouwer; Heyting) algebras are also known as `double brouwerian algebras', `semiboolean algebras', `Heyting--Brouwer algebras', `double Heyting algebras, or `bi-Heyting algebras' (resp., `co-Heyting algebras' or `dual Heyting algebras'; `pseudocomplemented lattices' or `pseudoboolean algebras'). Both Heyting and Brouwer algebras were considered in \cite[\S4.C, \S4.S2]{Curry:1963} as two cases of `Skolem lattices', defined as lattices satisfying either \eqref{eqn.adjunction.heyting.implication} or \eqref{eqn.adjunction.coheyting.implication}. Strictly speaking, the assumption $1\in K$ (resp., $0\in K$) in a definition of a Heyting (resp., Brouwer) algebra is obsolete, since its existence and uniqueness follows from $x\himp x=:1$ (resp., $x\cohimp x=:0$) $\forall x\in K$.%
}
\begin{enumerate}[nosep, label=(\roman*)]
\item {}\textup{\cite[p. 161]{Ogasawara:1939}} \textup{\cite[Thm. 46]{Birkhoff:1942}} \textup{\cite[Thm. 9.5]{Certaine:1943}} a \emph{Heyting algebra} if{}f it is equipped with $\himp:K\times K\ra K$ such that there is an adjunction $\cdot\land x\dashv x\himp\cdot$ $\forall x\in K$, i.e.
\begin{equation}
x\land y\leq z\;\;\mbox{if{}f}\;\;x\leq y\himp z\;\;\forall x,y,z\in K;
\label{eqn.adjunction.heyting.implication}
\end{equation}
\item {}\textup{\cite[Def. 1.1]{McKinsey:Tarski:1946}} a \emph{Brouwer algebra} if{}f it is equipped with $\cohimp:K\times K\ra K$ such that there is an adjunction $\cdot\cohimp x\dashv x\lor\cdot$ $\forall x\in K$, i.e. 
\begin{equation}
x\cohimp y\leq z\;\;\mbox{if{}f}\;\;x\leq y\lor z\;\;\forall x,y,z\in K;
\label{eqn.adjunction.coheyting.implication}
\end{equation}

\item {}\textup{\cite[Thm. 7]{Skolem:1919}} a \emph{Skolem algebra} if{}f it is equipped with $\himp,\cohimp:K\times K\ra K$ satisfying  \eqref{eqn.adjunction.heyting.implication} and \eqref{eqn.adjunction.coheyting.implication}.
\end{enumerate}
\end{definition}

\begin{proposition}\label{prop.skolem.properties}
\begin{enumerate}[nosep, label=(\roman*)]
\item If a lattice $(K,\leq,\land,\lor)$ is equipped with $\himp:K\times K\ra K$ (resp., $\cohimp:K\times K\ra K$) such that \eqref{eqn.adjunction.heyting.implication} (resp. \eqref{eqn.adjunction.coheyting.implication}) holds, then it is distributive.\label{prop.skolem.properties.1}
\item If $(K,\leq,\land,\lor,0,1)$ is a Heyting algebra, then $\hnot:K\ra K$, defined by \textup{\cite[Eqn. (L8)]{Ogasawara:1939}} $\hnot x:=(x\himp0)\;\;\forall x\in K,$
is an intuitionistic negation.\label{prop.skolem.properties.hnot} 
\item If $(K,\leq,\land,\lor,0,1)$ is a Brouwer algebra, then $\cohnot:K\ra K$, defined by \textup{\cite[Def. 1.2]{McKinsey:Tarski:1946}} $\cohnot x:=(1\cohimp x)\;\;\forall x\in K,$ is a cointuitionistic negation.\label{prop.skolem.properties.cohnot} 
\end{enumerate}
\end{proposition}

\begin{proof}
\begin{enumerate}[nosep, label=(\roman*)]
\item \cite[Thm. 6$_\times$]{Skolem:1919} (resp., \cite[Thm. 6$_+$]{Skolem:1919}). 
\item \ref{def: negation 1}--\ref{def: negation 2}: \cite[p. 148 (2nd ed.)]{Birkhoff:1940}; \ref{def: negation 6}: \cite[p. 161]{Ogasawara:1939}; \ref{def: negation 4}: follows from \eqref{eqn.adjunction.heyting.implication} via Lemma \ref{lemm.coheyting.negation}.\ref{item: lemm.coheyting.negation.heyting}.
\item \ref{def: negation 1}, \ref{def: negation 3}, \ref{def: negation 7}: \cite[Thms. 1.3.(iv), 1.3.(xii), 1.3.(xv)]{McKinsey:Tarski:1946}; \ref{def: negation 5}: follows from \eqref{eqn.adjunction.coheyting.implication} via Lemma \ref{lemm.coheyting.negation}.\ref{item: lemm.coheyting.negation.coheyting}.
\end{enumerate}
\end{proof}

\begin{definition}\label{def.coquasiintuitionist.Akchurin.etc}
A bounded Vakarelov algebra $(K,\leq,\land,\lor,0,1,{}^{\neg})$ will be called:
\begin{enumerate}[nosep, label=(\roman*)]
\item an \emph{orthocomplemented lattice} \textup{\cite[p. 830]{Birkhoff:vonNeumann:1936}} if{}f $^\neg$ is an orthocomplementation;\label{def.coquasiintuitionist.Akchurin.etc 1}
\item an \emph{orthomodular lattice} \textup{\cite[p. 780]{Husimi:1937}} if{}f it is an orthocomplemented lattice and (if $x\leq y$, then $y=x\lor (y\land x^{\neg})$ $\forall x,y\in K$).
\item a \emph{boolean algebra}\footnote{This definition has been originally a proposition, obtained in \cite[p. 460]{Birkhoff:1933} (the historically preceding equivalent definitions are discussed in details in \cite{Huntington:1904}). We will use the notion \textit{boolean lattice} as a synonym.} if{}f it is an orthocomplemented lattice and $(K,\leq,\land,\lor)$ is distributive;\label{def.coquasiintuitionist.Akchurin.etc 2}
\item a \emph{quasiintuitionistic} (resp., \emph{coquasiintuitionistic}) \emph{algebra} if{}f ${}^\neg$ is a quasiintuitionistic (resp., coquasiintuitionistic) negation and $(K,\leq,\land,\lor)$ is distributive;\label{def.coquasiintuitionist.Akchurin.etc 3}
%\item a \emph{normal} (resp., \emph{conormal}) algebra if{}f ${}^\neg$ is a normal (resp., conormal) negation;
\item a \emph{Skolem--Vakarelov} (resp., \emph{Heyting--Vakarelov}; \emph{Brouwer--Vakarelov}) \emph{algebra} if{}f $(K,\leq,\land,\lor,0,1)$ is a Skolem (resp., Heyting; Brouwer) algebra;\label{def.coquasiintuitionist.Akchurin.etc 7}
%\item \textup{\cite[p. 59]{Vakarelov:1976}} \textup{\cite[p. 165]{Berman:1977}} an \emph{Ockham algebra} if{}f ${}^\neg$ is biregular and $(K,\leq,\land,\lor)$ is distributive;\label{def.subminimal.DeMorgan.Ockham}
\item a \ss\emph{--quasiintuitionistic} (resp., \ss\emph{--coquasiintuitionistic}) \emph{algebra}, where \ss\ $\in\{$Brouwer, Heyting, Skolem$\}$, if{}f $(K,\leq,\land,\lor)$ is a \ss\ algebra, and $^\neg$ is quasiintuitionistic (resp., coquasiintuitionistic).
\end{enumerate}
\end{definition}

\begin{corollary}\label{cor.sum.of.cqi.qi.properties}
If $(K,\leq,\land,\lor,0,1,{}^\neg)$ is a coquasiintuitionistic (resp., quasiintuitionistic) algebra, then $^\neg$ satisfies \ref{def: negation 1}, \ref{def: negation 3}, \ref{def: negation 7}, \ref{def: negation 8}, \ref{def: negation 10}, \ref{def: negation 11}, \ref{def: negation 12}, \ref{def: negation 14}, \ref{def: negation 15}, \ref{def: negation 17}, and \ref{def: negation 18} (resp., \ref{def: negation 1}, \ref{def: negation 2}, \ref{def: negation 6}, \ref{def: negation 8}, \ref{def: negation 9}, \ref{def: negation 10}, \ref{def: negation 11}, \ref{def: negation 13}, \ref{def: negation 16}, \ref{def: negation 17}, and \ref{def: negation 18}).
\end{corollary}

\begin{proof}
Follows from Lemma \ref{lemm.bullet.negation.properties}.\ref{lemm.bullet.negation.properties 4}--\ref{lemm.bullet.negation.properties 5}, except of \ref{def: negation 15} (resp., \ref{def: negation 16}), following from boundedness of $(K,\leq,\land,\lor,0,1)$.
\end{proof}

\begin{corollary}\label{cor.coquasiintuitionistic.cointuitionistic}
\sloppy A coquasiintuitionistic (resp., quasiintuitionistic) algebra $(K,\leq,\land,\lor,0,1,{}^\neg)$ is a Brou\-wer (resp., Heyting) algebra if{}f ${}^\neg$ satisfies \ref{def: negation 5} (resp., \ref{def: negation 4}). A biquasiintuitionistic algebra $(K,\leq,\land,\lor,0,1,{}^\coeva,{}^\eva)$ is a Skolem algebra if{}f ${}^\eva$ satisfies \ref{def: negation 5}  and ${}^\coeva$ satisfies \ref{def: negation 4}.
\end{corollary}

\begin{proof}
The only difference between the equational properties of negations in coquasiintuitionistic algebra and Brouwer algebra is given by \ref{def: negation 5}, i.e. by the condition \ref{item: lemm.coheyting.negation 2} in Lemma \ref{lemm.coheyting.negation}, which is satisfied in a Brouwer algebra, and not satisfied in coquasiintuitionistic algebra. However, the equivalent condition \ref{item: lemm.coheyting.negation 1} in Lemma \ref{lemm.coheyting.negation} is exactly the definition of the negation $\cohnot$ by Proposition \ref{prop.skolem.properties}.\ref{prop.skolem.properties.cohnot}. Hence, $K$ satisfies \ref{item: lemm.coheyting.negation 2} if{}f it satisfies \ref{item: lemm.coheyting.negation 1} if{}f it is a Brouwer algebra. The result for quasiintuitionistic algebra follows by $(\land,\lor)$-duality, and this implies the result for biquasiintuitionistic algebra as well.
\end{proof}

\begin{corollary}\label{cor.demorgan.boolean}
If $(K,\leq,\land,\lor,0,1,{}^{\neg})$ is a De Morgan algebra, and ${}^{\neg}$ satisfies \ref{def: negation 6} or \ref{def: negation 7}, then $(K,\leq,\land,\lor,0,1,{}^{\neg})$ is a boolean lattice.
\end{corollary}

\begin{proof}
By Lemma \ref{lemm.bullet.negation.properties}, ${}^{\neg}$ satisfies \ref{def: negation 9}--\ref{def: negation 12}, since it satisfies \ref{def: negation 1}, \ref{def: negation 2}, and \ref{def: negation 3}. Assume that ${}^{\neg}$ satisfies \ref{def: negation 6}. Then $x \lor x^{\neg} = (x\lor x^{\neg})^{\neg\neg} = (x^{\neg} \land x^{\neg\neg})^{\neg} = 0^{\neg} = 1$ $\forall x\in K$, so ${}^{\neg}$ satisfies \ref{def: negation 7}, so ${}^{\neg}$ is an orthocomplementation. Hence $(K,\leq,\land,\lor,{}^{\neg})$ is a boolean lattice, since $(K,\leq,\land,\lor,0,1,{}^{\neg})$ is distributive. Assume that ${}^{\neg}$ satisfies \ref{def: negation 7}. By the same argument as above it is sufficient to show $x \land x^{\neg} = (x\land x^{\neg})^{\neg\neg} = (x^{\neg} \lor x^{\neg\neg})^{\neg} = 1^{\neg} = 0$ $\forall x\in K$.
\end{proof}

\begin{definition}
If $(K,\leq,\land,\lor)$ is a lattice and ${}^\eva,{}^\coeva:K\ra K$, then a structure $(K,\leq,\land,\lor,{}^\coeva,{}^{\eva})$ will be called:
\begin{enumerate}[nosep, label=(\roman*)]
\item a \emph{bi-Vakarelov algebra} if{}f ${}^\eva$ and ${}^\coeva$ are subminimal;
%\item a \emph{biregular algebra} if{}f ${}^\eva$ is coregular and ${}^\coeva$ is regular;
\item a \emph{biquasiintuitionistic algebra} if{}f ${}^\eva$ is coquasiintuitionistic and ${}^\coeva$ is quasiintuitionistic.
\end{enumerate}
\end{definition}

\begin{definition}\label{def.biquasiintuitionistic}
A bounded bi-Vakarelov algebra $(K,\leq,\land,\lor,0,1,{}^\coeva,{}^\eva)$ will be called:
\begin{enumerate}[nosep, label=(\roman*)]
\item a \emph{biquasiintuitionistic} \emph{algebra} if{}f $(K,\leq,\land,\lor,0,1,{}^{\eva})$ is a coquasiintuitionistic algebra and $(K,\leq,\land,\lor,0,1,{}^{\coeva})$ is a quasiintuitionistic algebra;
\item\label{def.biquasiintuitionistic.2} an \emph{Akchurin algebra} if{}f $(K,\leq,\land,\lor,0,1,{}^{\coeva})$ is a quasiintuitionistic algebra, $(K,\leq,\land,\lor,0,1,{}^{\eva})$ is a coquasiintuitionistic algebra, and $(K,\leq,\land,\lor,0,1)$ is a Skolem algebra.
\end{enumerate}
\end{definition}

\begin{definition}\label{def: O6}
The orthocomplemented lattice $(\mathcal{O}_6,{ }^{\perp})$ is given by $\mathcal{O}_6:=\{0,x,y,x^{\perp},y^{\perp},1\}$ such that $x\leq y$. 
\end{definition}

\begin{proposition}\label{prop: OML fun}
\begin{enumerate}[nosep, label=(\roman*)]
\item An orthocomplemented lattice $(L,{ }^{\perp})$ is orthomodular if{}f $\mathcal{O}_6$ is not a sublattice of $(L,{ }^{\perp})$.\label{item: OML fun 0}
\item An orthocomplemented lattice $(L,{ }^{\perp})$ is orthomodular if{}f $\forall a\in L\setminus\{0,1\}$ $\nexists b\in L\setminus\{0,1\}$ $\st$ $a^{\perp} \leq b$ and $a\land b =0$. \label{item: OML fun 1}
\end{enumerate}
\end{proposition}

\begin{proof}
\begin{enumerate}[nosep, label=(\roman*)]
\item {} \textup{\cite[Thm. 1]{Maeda:1966}}.
\item Due to \ref{item: OML fun 0}, it is sufficient to show that, under the given assumptions, $\{0,a^{\perp},a,b,b^{\perp},1\}$ forms $(\mathcal{O}_6,{}^\perp)$ as a sublattice of $(L,{}^\perp)$. By Definition \ref{def: O6}, it is sufficient to verify $a^{\perp}\leq b$, $a^{\perp}\land b^{\perp} = 0,$ and $a\land b = 0$. The first inequality and the last equality are given by the assumptions. For the remaining equality, $1 = a\lor a^{\perp} \leq a\lor b$ implies $a \lor b = 1$, hence $a^{\perp} \land b^{\perp} = 0$.
\end{enumerate}
\end{proof}

\begin{definition}\label{def.Stone.space}
\begin{enumerate}[nosep, label=(\roman*)]
\item A map $f:L\ra K$ between two orthocomplemented (resp., boolean) lattices $L$ and $K$ is called: an \emph{orthocomplemented} (resp., \emph{boolean}) \emph{lattice homomorphism} if{}f it preserves meets, joins, and orthocomplementations (resp., meets, joins, top element, and bottom element); an \emph{orthocomplemented} (resp., \emph{boolean}) \emph{lattice isomorphism}, denoted by $\iso$, if{}f it is an orthocomplemented (resp., boolean) lattice homomorphism and there exists an orthocomplemented (resp., boolean) lattice homomorphism $g:K\ra L$ such that $g\circ f=\id_K$ and $f\circ g=\id_L$.
\item\textup{\cite[Thm. III]{Stone:1934}} The \emph{Stone spectrum} of a boolean lattice $(B,\leq,\land,\lor,0,1,{}^\perp)$, denoted by $\sps(B)$, is defined as the set of all boolean lattice homomorphisms.
\item A totally disconnected, compact, Hausdorff topological space will be called a \emph{Stone space}.
\end{enumerate}
\end{definition}

\begin{proposition}\label{thm: Stone}\textup{\cite[Thms. IV$_1$, IV$_2$]{Stone:1934} \cite[Thms. 1, 3, 4]{Stone:1937}}
If $(B,\leq,\land,\lor,0,1,{}^\perp)$ is a boolean lattice, then:
\begin{enumerate}[nosep, label=(\roman*)]
\item\label{thm: Stone.1}$\sps(B)$ is a Stone space with respect to a topology generated by a basis given by $\{\{x\in\sps(B)\mid x(y)=1\}\mid y\in B\}$;
\item\label{thm: Stone.2} $B$ is isomorphic to the lattice $\mathrm{clop}\left( \sps(B) \right)$ of closed-and-open sets of $\sps(B)$ under the above topology. The corresponding lattice isomorphism is given by 
\begin{equation}
\stone_{B}:B\ni a \mapsto \stone_{B}(a) \coloneqq \lbrace \xi \in \sps(B)\mid \xi(a) = 1\rbrace \in \mathrm{clop}(\sps(B)),
\end{equation}
with meets (resp., joins) mapped to the {interior of} set theoretic intersections (resp., {closure of set theoretic} sums), and $\stone_B(a^\perp)=\sps(B)\setminus\stone_{B}(a)$ $\forall a\in B$.
\end{enumerate}
\end{proposition}

\section{\label{section.internal.lattice.akchurin}Internal lattices of Akchurin algebras}

\begin{proposition}\label{prop: CQI Internal}
Let $(\extlattice,\leq,\land,\lor,0,1,{}^{\eva})$ be a coquasiintuitionistic algebra and let $\evaintlatticeext\coloneqq \{x^{\upint} \mid x\in \extlattice\} \subseteq \extlattice$ be equipped with the restriction of the partial order $\leq$ on $\extlattice$ to $\evaintlatticeext$ and the restriction of ${}^\eva$ on $\extlattice$ to $\evaintlatticeext$, denoted without change of symbol. Then $(\evaintlatticeext,\leq,\curlywedge,\lor,0,1,{}^\eva)$ is an orthocomplemented lattice such that:
\begin{enumerate}[nosep, label=(\roman*)]
\item\label{prop: CQI Internal.i} the join $\lor$ on $\evaintlatticeext$ is given by the join $\lor$ of $\extlattice$ restricted to $\evaintlatticeext$;
\item\label{prop: CQI Internal.ii} the meet $\curlywedge$ on $\evaintlatticeext$ is given by $\alpha \curlywedge \beta \coloneqq (\alpha \land \beta)^{\eva \eva} $ $\forall \alpha,\beta \in \evaintlatticeext$;
\item\label{prop: CQI Internal.iii} the bottom (resp., top) element of $\evaintlatticeext$ is given by $0$ (resp., $1$);
\item\label{prop: CQI Internal.iv} the orthocomplementation is given by ${}^{\eva}$.
\end{enumerate}
\end{proposition}
\begin{proof}
Let $\alpha,\beta\in \evaintlatticeext$ and $x,y\in \extlattice$ be such that $\alpha = x^{\upint}$ and $\beta = y^{\upint}$. 
\begin{enumerate}[nosep, label=(\roman*)]
\item Since $\lor$ is the join on $\extlattice$, it is sufficient to show that $\alpha \lor \beta \in \evaintlatticeext$, since then $\lor$ is necessarily the join on $\evaintlatticeext$, due to $\evaintlatticeext\subseteq \extlattice$. $\alpha\lor\beta = x^{\upint} \lor y^{\upint} = (x^{\eva} \land y^{\eva})^{\eva} = (x^{\eva} \land y^{\eva})^{\eva\upint} = z^{\upint}$, using \ref{def: negation 8}, \ref{def: negation 11}, and \ref{def: negation 12} by Corollary \ref{cor.sum.of.cqi.qi.properties}, and defining $z\coloneqq (x^{\eva} \land y^{\eva})^{\eva}\in \extlattice$.
\item Clearly $(\alpha\land \beta)^{\eva \eva} \in\evaintlatticeext$, since $\alpha\land\beta = x^{\upint} \land y^{\upint} \in \extlattice$. Thus, it is sufficient to show that $(\alpha\land \beta)^{\eva \eva}$ is the largest element smaller than both $\alpha$ and $\beta$. Assume that $\exists\gamma\in\evaintlatticeext$ such that $\gamma \leq \alpha,\  \gamma\leq \beta$, and $(\alpha\land \beta)^{\eva \eva}\leq \gamma$. Using \ref{def: negation 1}, \ref{def: negation 8}, \ref{def: negation 11}, and \ref{def: negation 12}, by Corollary \ref{cor.sum.of.cqi.qi.properties}, $\gamma^{\eva} \leq (\alpha\land \beta)^{\eva\eva\eva} = (\alpha\land \beta)^{\eva} = \alpha^{\eva} \lor \beta^{\eva}$, with $\gamma^{\eva} \geq \alpha^{\eva}$ and $\gamma^{\eva}\geq \beta^{\eva}$, so $\gamma^{\eva} = \alpha^{\eva} \lor \beta^{\eva} = (\alpha \land \beta)^{\eva}$, by \ref{prop: CQI Internal.i}. By \ref{prop: CQI Internal.iv}, $\gamma^{\eva\eva} = \gamma$, so $\gamma = \gamma^{\eva\eva} = (\alpha \land \beta)^{\upint} = \alpha \curlywedge \beta$. 
\item By Corollary \ref{cor.sum.of.cqi.qi.properties}, ${}^{\eva}$ satisfies \ref{def: negation 17} and \ref{def: negation 18}, hence $0 = 1^{\eva} = 0^{\upint}$, so $0\in \evaintlatticeext$. Thus $\alpha \curlywedge 0 = (\alpha \land 0)^{\upint} = (\alpha)^{\upint} = \alpha$  $\forall \alpha \in \evaintlatticeext$, by \ref{prop: CQI Internal.iv} and since $0$ is the bottom element in $\extlattice$. So $0  \leq \alpha$ $\forall \alpha\in \evaintlatticeext$. Analogously, $1 = 0^{\eva} = 1^{\upint}$, hence $1\in \evaintlatticeext$. Thus $\alpha \lor 1 = 1$ $\forall \alpha\in\evaintlatticeext$, since $\lor$ is inherited from $\extlattice$ and $1$ is the top element in $\extlattice$, so $\alpha \leq 1$ $\forall \alpha\in\evaintlatticeext$. 
\item Due to Lemma \ref{lemm.bullet.negation.properties}.\ref{lemm.bullet.negation.properties y}, it is sufficient to show that $\alpha \curlywedge \alpha^{\eva} = 0$ and $\alpha^{\eva \eva} = \alpha$. $\alpha \curlywedge \alpha^{\eva} = (\alpha \land \alpha^{\eva})^{\upint} = (\alpha^{\eva} \lor \alpha^{\upint})^{\eva} = 1^{\eva} = 0$ by \ref{def: negation 11}, \ref{def: negation 12}, and \ref{def: negation 17}, due to Corollary \ref{cor.sum.of.cqi.qi.properties} and $\alpha \in \evaintlatticeext\subseteq \extlattice$. $\alpha^{\eva\eva} = x^{\upint\eva\eva} = x^{\upint} = \alpha $ by \ref{def: negation 8}, due to Corollary \ref{cor.sum.of.cqi.qi.properties}.
\end{enumerate}
\end{proof}

\begin{proposition}\label{prop: Internal completeness}
\sloppy If $(\extlattice,\leq,\land,\lor,0,1,{}^\eva)$ is a complete coquasiintuitionistic algebra, then $(\evaintlatticeext,\leq,\curlywedge,\lor,0,1,{}^\eva)$ is complete.
\end{proposition}

\begin{proof}
Let $I$ be a set and let $M\coloneqq \{\alpha_{i}\in \evaintlatticeext\mid i\in I\}$ be a nonempty subset of $\evaintlatticeext$. Let $x_{i} \in \extlattice$ be such that $(x_{i})^{\eva\eva} = \alpha_{i}\in M$ $\forall i\in I$.
\begin{enumerate}[nosep, label=(\roman*)]
\item $(\bigvee M)^{\eva\eva} = (\bigvee_{i\in I} \alpha_{i})^{\eva\eva} = (\bigvee_{i\in I} (x_{i})^{\eva\eva})^{\eva\eva} = (\bigwedge_{i\in I} (x_{i})^{\eva})^{\eva\eva\eva} =  (\bigwedge_{i\in I} (x_{i})^{\eva})^{\eva} = \bigvee_{i\in I} (x_{i})^{\eva\eva} = \bigvee_{i\in I} \alpha_{i} = \bigvee M$, by \ref{def: negation 8}, \ref{def: negation 11}, and \ref{def: negation 12}, due to Corollary \ref{cor.sum.of.cqi.qi.properties}. Thus, the join in $\extlattice$ of a subset $M$ of $\evaintlatticeext$ is contained in $\evaintlatticeext$.
\item $\bigwedge M\leq x_i$ $\forall i\in I$, so $\bigcurlywedge M:=(\bigwedge M)^{\eva\eva}\leq(x_i)^{\eva\eva}=\alpha_i$ $\forall i\in I$, by \ref{def: negation 1}. Assume that $\exists \gamma \in \evaintlatticeext$ such that $(\bigwedge M)^{\eva\eva} \leq \gamma \leq \alpha_{i}$ $\forall i\in I$. Then $(\bigwedge M)^{\eva\eva\eva} \geq \gamma^{\eva} \geq (\alpha_{i})^{\eva}$ $\forall i\in I$, by \ref{def: negation 1}. From $(\alpha_{i})^{\eva} = (x_{i})^{\eva\eva\eva} = (x_{i})^{\eva}$, by \ref{def: negation 8}, and $(\bigwedge M)^{\eva\eva\eva} = (\bigwedge_{i\in I} (x_{i})^{\eva\eva})^{\eva\eva\eva} = (\bigwedge_{i\in I} (x_{i})^{\eva\eva})^{\eva} = \bigvee_{i\in I} (x_{i})^{\eva\eva\eva} = \bigvee_{i\in I} (x_{i})^{\eva}$, by \ref{def: negation 11} and \ref{def: negation 12}, it follows that $\bigvee_{i\in I} (x_{i})^{\eva} = (\bigwedge M)^{\eva\eva\eva} \geq \gamma^{\eva} \geq (\alpha_{i})^{\eva} = (x_{i})^{\eva}$ $\forall i\in I$. So $(\bigwedge M)^{\eva\eva\eva} = \gamma^{\eva}$, by definition of the join in $\extlattice$. Hence, by Proposition \ref{prop: CQI Internal}.\ref{prop: CQI Internal.iv} and \ref{def: negation 8}, $\gamma = \gamma^{\eva\eva} = (\bigwedge M)^{\eva\eva\eva\eva} =(\bigwedge M)^{\eva\eva}$. Thus, the meet of a subset $M$ of $\evaintlatticeext$ exists and is given by $\bigcurlywedge M$.
\end{enumerate} 
\end{proof}

\begin{corollary}\label{cor: QI Internal}
Let $(\extlattice,\leq,\land,\lor,0,1,{}^{\coeva})$ be a quasiintuitionistic algebra and let $\coevaintlatticeext\coloneqq \{x^{\coupint} \mid x\in \extlattice\} \subseteq \extlattice$ be equipped with the restriction of the partial order $\leq$ on $\extlattice$ to $\coevaintlatticeext$ and the restriction of ${}^\coeva$ on $\extlattice$ to $\coevaintlatticeext$, denoted without change of symbol. Then $(\coevaintlatticeext,\leq,\land,\curlyvee,0,1,{}^\coeva)$ is an orthocomplemented lattice such that:
\begin{enumerate}[nosep, label=(\roman*)]
\item\label{prop: QI Internal.i} the join $\curlyvee$ on $\coevaintlatticeext$ is given by $\alpha \curlyvee \beta \coloneqq (\alpha \lor \beta)^{\coupint} $ $\forall \alpha,\beta \in \coevaintlatticeext$;
\item\label{prop: QI Internal.ii} the meet $\land$ on $\coevaintlatticeext$ is given by the meet $\land$ of $\extlattice$ restricted to $\coevaintlatticeext$;
\item\label{prop: QI Internal.iii} the bottom (resp., top) element of $\coevaintlatticeext$ is given by $0$ (resp., $1$)$\in \extlattice$;
\item\label{prop: QI Internal.iv} the orthocomplementation is given by ${}^{\coeva}$.
\end{enumerate}
\end{corollary}

\begin{proof}
Follows from Proposition \ref{prop: CQI Internal} by $(\land,\lor)$-duality.
\end{proof}

\begin{corollary}\label{cor: QI Internal completeness}
\sloppy If $(\extlattice,\leq,\land,\lor,0,1,{}^\coeva)$ is a complete quasiintuitionistic algebra, then $(\coevaintlatticeext,\leq,\land,\,\curlyvee,0,1,{}^\coeva)$ is complete.
\end{corollary}
\begin{proof}
Follows from Proposition \ref{prop: Internal completeness} by $(\land,\lor)$-duality.
\end{proof}

\begin{proposition}\label{prop.brouwer.heyting.internal.boolean}
If $(\extlattice,\leq,\land,\lor,0,1,{}^\eva)$ (resp., $(\extlattice,\leq,\land,\lor,0,1,{}^\coeva)$) is a Brouwer (resp., Heyting) algebra with $(\cdot)^{\eva}=\cohnot(\cdot)$ (resp., $(\cdot)^{\coeva}=\hnot(\cdot)$), then $(\evaintlatticeext,\leq,\curlywedge,\lor,0,1,{}^\coeva)$ (resp., $(\coevaintlatticeext,\leq,\land,\,\curlyvee,0,1,{}^\coeva)$) is a boolean algebra.
\end{proposition}

\begin{proof}
For $(K,\leq,\land,\lor,0,1,{}^\eva)$ this was proved in \cite[Thm. 4.5]{McKinsey:Tarski:1946}. The statement for $(K,\leq,\land,\lor,0,1,{}^\coeva)$ follows by $(\land,\lor)$-duality.
\end{proof}

\begin{corollary}\label{cor.akchurin.internal.lattices}
If $(K,\leq,\land,\lor,0,1,\hnot,\himp,\cohnot,\cohimp,{}^\coeva,{}^\eva)$ is an Akchurin algebra, then:
\begin{enumerate}[nosep, label=(\roman*)]
\item\label{cor.akchurin.internal.lattices.i} $(\evaintlatticeext,\leq,\curlywedge,\lor,0,1,{}^\coeva)$ and $(\coevaintlatticeext,\leq,\land,\,\curlyvee,0,1,{}^\coeva)$ are orthocomplemented lattices;
\item\label{cor.akchurin.internal.lattices.ii} $(K_{\hnot\hnot},\leq,\curlywedge,\lor,0,1,\hnot)$ and $(K_{\cohnot\cohnot},\leq,\land,\,\curlyvee,0,1,\cohnot)$ are boolean lattices;
\item\label{cor.akchurin.internal.lattices.iv} $(\evaintlatticeext,\leq,\curlywedge,\lor,0,1,{}^\eva)\iso(\coevaintlatticeext,\leq,\land,\,\curlyvee,0,1,{}^\coeva)$ if $\Theta(x^\coeva)=(\Theta(x))^\eva$;
\item\label{cor.akchurin.internal.lattices.iii} $(K_{\hnot\hnot},\leq,\curlywedge,\lor,0,1,\hnot)\iso(K_{\cohnot\cohnot},\leq,\land,\,\curlyvee,0,1,\cohnot)$.
\end{enumerate}
\end{corollary}

\begin{proof}
\begin{enumerate}[nosep, label=(\roman*)]
\item Follows from Definition \ref{def.biquasiintuitionistic}.\ref{def.biquasiintuitionistic.2}, Proposition \ref{prop: CQI Internal}, and Corollary \ref{cor: QI Internal}.
\item Follows from Proposition \ref{prop.brouwer.heyting.internal.boolean}.
\item Follows from Definition \ref{def.biquasiintuitionistic}.\ref{def.biquasiintuitionistic.2}.
\item Follows from \ref{cor.akchurin.internal.lattices.iv} and Definition \ref{def.heyting.brouwer.skolem.algebra}.
\end{enumerate}
\end{proof}

\begin{proposition}\label{prop: Precqi internal oml}
If $(\extlattice,\leq,\land,\lor,0,1,{}^\eva)$ is a coquasiintuitionistic algebra, then $(\evaintlatticeext,\leq,\curlywedge,\lor,0,1,{}^\eva)$ is an orthomodular lattice if{f} $\forall x\in \extlattice\setminus\{0,1\}\ \nexists y\in \extlattice\setminus\{0,1\}\st x^{\eva} \leq y $ and $(x\land y)^\eva=1$.
\end{proposition}

\begin{proof}
By Proposition \ref{prop: OML fun}.\ref{item: OML fun 1}, $(\evaintlatticeext,\leq,\curlywedge,\lor,0,1,{}^\eva)$ is an orthomodular lattice if{f} 
\begin{equation}
\forall \alpha\in \evaintlatticeext\setminus\{0,1\}\ \nexists \beta\in \evaintlatticeext\setminus\{0,1\}\st \alpha^{\eva} \leq \beta \mbox{ and }\alpha\curlywedge\beta = 0.
\label{eqn.internal.lattice.orthomodular.condition}
\end{equation}
By definition of $\evaintlatticeext$, $\exists x,y\in K\st \alpha=x^{\upint}$ and $\beta= y^{\upint}$, hence, by definition of $\curlywedge$, this condition is equivalent to $\forall x\in \extlattice\setminus\{0,1\}\ \nexists y\in \extlattice\setminus\{0,1\}\st x^{\eva} \leq y^{\upint} $ and $(x\land y)^{\upint} = 0$. $x^{\eva} \leq y^{\upint}$ if{f} $x^{\eva} \leq y$, since (if $x^{\eva}\leq y^{\upint}$, then $x^{\eva}\leq y^{\upint}\leq y$), by \ref{def: negation 3}. If $x^{\eva} \leq y$, then $x^{\eva} = x^{\eva\upint} \leq y^{\upint}$, since ${}^{\eva}$ satisfies \ref{def: negation 1} and \ref{def: negation 8}, and $(\cdot)^{\eva\eva}$ is order monotone. $(x\land y)^{\upint} = 0$ if{}f $(x\land y)^\eva=1$. Thus \eqref{eqn.internal.lattice.orthomodular.condition} is equivalent to $\forall x\in \extlattice\setminus\{0,1\}\ \nexists y\in \extlattice\setminus\{0,1\}\st x^{\eva} \leq y $ and $(x\land y)^\eva=1$.
\end{proof}

\begin{corollary}\label{cor.quasiintuitionistic.internal.orthomodular.lattice}
\sloppy If $(\extlattice,\leq,\land,\lor,0,1,{}^\coeva)$ is a quasiintuitionistic algebra, then $(\coevaintlatticeext,\leq,\curlywedge,\lor,0,1,{}^\coeva)$ is an orthomodular lattice if{f} $\forall x\in \extlattice\setminus\{0,1\}\ \nexists y\in \extlattice\setminus\{0,1\}\st y\leq x^{\coeva} $ and $(x\lor y)^\coeva=0$.
\end{corollary}

\begin{proof}
Follows from Proposition \ref{prop: Precqi internal oml} by $(\land,\lor)$-duality.
\end{proof}

\section{\label{section.quasiintuitionistic.Akchurin.logics}(Co/bi/)quasiintutionistic and Akchurin logics}

\begin{definition}\label{def.coquasiintuitionistic.logic}
\begin{enumerate}[nosep, label=(\roman*)]
\item \label{def.coquasiintuitionistic.logic.1}Let $\L_0$ be a propositional language, consisting of a set of propositions and binary connectives $\land$ and $\lor$. Let $\L^{\boti,\topi}$ denote the extension of $\L_0$ with propositional constants $\bot$ and $\top$. Let $\L_{\qineg}$ (resp., $\L_{\cqineg}$; $\L_{\qineg,\cqineg}$) denote the extension of $\L_0$ with unary connective $\qineg$ (resp., a unary connective $\cqineg$; unary connectives $\qineg$ and $\cqineg$), and let $\L_{\qineg}^{\boti,\topi}$ (resp., $\L_{\cqineg}^{\boti,\topi}$; $\L_{\qineg,\cqineg}^{\boti,\topi}$) denote a corresponding extension of $\L^{\boti,\topi}$. Let $\L\in\{\L_0,\L_{\qineg},\L_{\cqineg},\L_{\qineg,\cqineg},\L^{\boti,\topi},\L^{\boti,\topi}_{\qineg},\L^{\boti,\topi}_{\cqineg},\L^{\boti,\topi}_{\qineg,\cqineg}\}$. The formul{\ae} of $\L$, denoted by $p,q,r,\ldots$, are given by propositions, propositional constants, and the results of application of connectives to them. Given formul{\ae} $p$ and $q$ of $\mathcal{L}$, the ordered pair $(p,q)$, denoted by $p\gentzen q$, will be called an \emph{$\mathcal{L}$-sequent}. The set of all formul{\ae} of $\mathcal{L}$ (resp., all $\mathcal{L}$-sequents) will be denoted by $\Frm(\mathcal{L})$ (resp., $\Seq(\mathcal{L})$). 
\item \label{def.coquasiintuitionistic.logic.2}Consider the following axioms and rules:
\vspace{0.1cm}
\begin{enumerate}[nosep, label=\textup{a\arabic*)}]
\begin{minipage}{0.5\linewidth}
\item $p\gentzen p$;\label{coqint a1}
\item $p\land q\gentzen p$;\label{coqint a2}
\item $p\land q\gentzen q$;\label{coqint a3}
\item $p\gentzen p\lor q$;\label{coqint a4}
\item $q\gentzen p\lor q$;\label{coqint a5}
\item $p\land(q\lor r)\gentzen(p\land q)\lor(p\land r)$;\label{coqint a6}
\item $\cqineg\cqineg p\gentzen p$;\label{coqint a7}
\item $q\gentzen p\lor\cqineg p$;\label{coqint a8}
\item $p\gentzen \qineg\qineg  p$;\label{coqint a9}
\item $p\land\qineg p\gentzen q$;\label{coqint a10}
\end{minipage}
\begin{minipage}{0.5\linewidth}
\item $\cqineg\top\gentzen\bot$;\label{coqint a11}
\item $\top\gentzen\qineg\bot$;\label{coqint a12}
\item $p\gentzen\top$;\label{coqint a13}
\item $\bot\gentzen p$;\label{coqint a14}
%\item $\cqineg(p\land q)\gentzen\cqineg p\lor\cqineg q$;\label{coqint a15}
%\item $\qineg p\land\qineg q\gentzen\qineg(p\lor q)$;\label{coqint a16}
\item[\hypertarget{coqint r1}{\textup{r1)}}]  if $p\gentzen q$ and $q\gentzen r$, then $p\gentzen r$;
\item[\hypertarget{coqint r2}{\textup{r2)}}] if $p\gentzen q$ and $p\gentzen r$, then $p\gentzen q\land r$;
\item[\hypertarget{coqint r3}{\textup{r3)}}] if $p\gentzen r$ and $q\gentzen r$, then $p\lor q\gentzen r$;
\item[\hypertarget{coqint r4}{\textup{r4)}}] if $p\gentzen q$, then $\cqineg q\gentzen\cqineg p$;
\item[\hypertarget{coqint r5}{\textup{r5)}}] if $p\gentzen q$, then $\qineg q\gentzen\qineg p$.
\end{minipage}
\end{enumerate}
\vspace{0.1cm}
The subset of $S$ consisting of the axioms $A$ and all sequents obtained from them by application of the rules $R$ and the rule of substitution will be called $N$, and will be denoted by $L$, where the quintuples $(S;A;R;N;L)$ are given by:
\begin{enumerate}[nosep, label=\alph*)]
\item ($\Seq(\mathcal{L}^{\boti,\topi}_{\cqineg})$; \ref{coqint a1}--\,\ref{coqint a8}, \ref{coqint a11}, \ref{coqint a13}, and \ref{coqint a14}; \hyperlink{coqint r1}{\textup{r1)}}--\,\hyperlink{coqint r4}{\textup{r4)}}; a \emph{coquasiintuitionistic logic}; ${\coQInt}$);
\item ($\Seq(\mathcal{L}^{\boti,\topi}_{\qineg})$; \ref{coqint a1}--\,\ref{coqint a6}, \ref{coqint a9}, \ref{coqint a10}, and \ref{coqint a12}--\,\ref{coqint a14}; \hyperlink{coqint r1}{\textup{r1)}}--\,\hyperlink{coqint r3}{\textup{r3)}} and \hyperlink{coqint r5}{\textup{r5)}}; a \emph{quasiintuitionistic logic}; $\QInt$);
\item ($\Seq(\mathcal{L}^{\boti,\topi}_{\cqineg, \qineg})$; \ref{coqint a1}--\,\ref{coqint a14}; \hyperlink{coqint r1}{\textup{r1)}}--\,\hyperlink{coqint r5}{\textup{r5)}}; a \emph{biquasiintuitionistic logic}; $\biQInt$).
\end{enumerate}
\item \label{def.coquasiintuitionistic.logic.3}A \emph{valuation} $\Chi_K$ of $\Frm(\mathcal{L}^{\boti,\topi}_{\cqineg})$ in a structure $(K,\leq,\land,\lor,0,1,{}^\eva)$, where $K$ is a set, $\leq$ is a relation on $K$, $\land,\lor:K\times K\ra K$, ${}^\eva:K\ra K$, and $0,1\in K$, is defined by $\Chi_K(p\lor q):=\Chi_K(p)\lor\Chi_K(q)$, $\Chi_K(p\land q):=\Chi_K(p)\land\Chi_K(q)$, $\Chi_K(\bot):=0$, $\Chi_K(\top):=1$, $\Chi_K(\cqineg p):=(\Chi_K(p))^\eva$ $\forall p,q\in\Frm(\mathcal{L}^{\boti,\topi}_{\cqineg})$. A sequent $p\gentzen q\in\coQInt$ is said to be: \emph{satisfied} in a valuation $\Chi_K$ in $(K,\leq,\land,\lor,0,1,{}^\eva)$ if{}f $\Chi_K(p)\leq\Chi_K(q)$; \emph{valid} in $(K,\leq,\land,\lor,0,1,{}^\eva)$ if{}f it is satisfied in all valuations in $(K,\leq,\land,0,1,{}^\eva)$. $\coQInt$ will be called \emph{sound} (resp., \emph{complete}) with respect to a subclass $\mathcal{C}$ of structures $(K,\leq,\land,\lor,0,1,{}^\eva)$ if{}f (if $p\gentzen q\in\coQInt$, then $p\gentzen q$ is valid in all members of $\mathcal{C}$) (resp., if $p\gentzen q\in\Seq(\mathcal{L}^{\boti,\topi}_{\cqineg})$ is valid in all members of $\mathcal{C}$, then $p\gentzen q\in\coQInt$). Corresponding definitions for $\QInt$ and $\biQInt$ are analogous.
\end{enumerate}
\end{definition}

\begin{proposition}\label{prop.soundness.of.coQInt}
\begin{enumerate}[nosep, label=(\roman*)]
\item \sloppy ${\coQInt}$ is sound and complete with respect to the class of coquasiintuitionistic algebras.\label{prop.soundness.of.coQInt 1}
\item ${\QInt}$ is sound and complete with respect to the class of quasiintuitionistic algebras.\label{prop.soundness.of.coQInt 2}
\item ${\biQInt}$ is sound and complete with respect to the class of biquasiintuitionistic algebras.\label{prop.soundness.of.coQInt 3}
\end{enumerate}
\end{proposition}

\begin{proof}
\begin{enumerate}[nosep, label=(\roman*)]
\item \begin{enumerate}[nosep, label=\alph*)]
\item Validity of \ref{coqint a1} (resp., \ref{coqint a2} and \ref{coqint a3}; \ref{coqint a4} and \ref{coqint a5}; \ref{coqint a6}; \ref{coqint a7}; \ref{coqint a8}; \ref{coqint a11}; \ref{coqint a13} and \ref{coqint a14}) in a coquasiintuitionistic algebra $(\extlattice,\leq,\land,\lor,0,1,{}^\eva)$ follows from reflexivity of $\leq$ (resp., definition of meets; definition of joins; distributivity; \ref{def: negation 3}; \ref{def: negation 7}; Lemma \ref{lemm.bullet.negation.properties}.\ref{lemm.bullet.negation.properties 45}; boundedness of $(\extlattice,\leq,\land,\lor,0,1)$), while validity of sequents derived by an application of \hyperlink{coqint r1}{\textup{r1)}} (resp., \hyperlink{coqint r2}{\textup{r2)}}; \hyperlink{coqint r3}{\textup{r3)}}; \hyperlink{coqint r4}{\textup{r4)}}) follows from transitivity of $\leq$ (resp., the universal property of the meet; the universal property of the join; \ref{def: negation 1}). By induction on the length of proof of the formul{\ae}, this establishes soundness of ${\coQInt}$ with respect to the class of coquasiintuitionistic algebras.
\item The proof of completeness is based on a construction of the Lindenbaum algebra (cf. \cite[pp. 122--123]{McKinsey:1941}) of $\coQInt$. Given $p,q\in\Frm(\mathcal{L}^{\boti,\topi}_{\cqineg})$, $p\iffgentzen q$, defined as a notation for $p\gentzen q,q\gentzen p\in\coQInt$, is an equivalence relation on $\Frm(\mathcal{L}^{\boti,\topi}_{\cqineg})$. Let $\bar{K}:=\Frm(\mathcal{L}^{\boti,\topi}_{\cqineg})/\!\!\iffgentzen$ (i.e. $\bar{K}:=\{\igequ{p}\mid p\in\Frm(\mathcal{L}^{\boti,\topi}_{\cqineg})\}$ with $\igequ{p}:=\{q\in\Frm(\mathcal{L}^{\boti,\topi}_{\cqineg})\mid q\iffgentzen p\}$), and define $0:=\igequ{\bot}$, $1:=\igequ{\top}$, as well as $\land,\lor:\bar{K}\times \bar{K}\ra \bar{K}$ and ${}^\eva:\bar{K}\ra \bar{K}$ by $\igequ{p}\land\igequ{q}:=\igequ{p\land q}$, $\igequ{p}\lor\igequ{q}:=\igequ{p\lor q}$, $\igequ{p}^\eva:=\igequ{\cqineg p}$ $\forall p,q\in\Frm(\mathcal{L}^{\boti,\topi}_{\cqineg})$. By \ref{coqint a4} (resp., \hyperlink{coqint r4}{\textup{r4)}} applied to \ref{coqint a1} and \ref{coqint a2}; \ref{coqint a2}; \hyperlink{coqint r2}{\textup{r2)}} applied to \ref{coqint a1} and \ref{coqint a4}), $p\gentzen p\lor(p\land q)$ (resp., $p\lor(p\land q)\gentzen p$; $p\land(p\lor q)\gentzen p$; $p\gentzen p\land(p\lor q)$), hence $\igequ{p}\lor(\igequ{p}\land\igequ{q})=\igequ{p}$ and $\igequ{p}\land(\igequ{p}\lor\igequ{q})=\igequ{p}$ $\forall\igequ{p},\igequ{q}\in \bar{K}$, so $(\bar{K},\land,\lor)$ is a lattice. Partial order $\leq$ on $\bar{K}$ is defined by $\igequ{p}\lor\igequ{q}=\igequ{q}$ if{}f $\igequ{p}\leq\igequ{q}$ $\forall\igequ{p},\igequ{q}\in \bar{K}$. Let $p,q,r\in\Frm(\mathcal{L}^{\boti,\topi}_{\cqineg})$. By \ref{coqint a1} and \hyperlink{coqint r3}{\textup{r3)}}: if $p\gentzen q$ and $q\gentzen q$, then $p\lor q\gentzen q$. By \ref{coqint a5}: $q\gentzen p\lor q$. By \ref{coqint a4} and \hyperlink{coqint r1}{\textup{r1)}}: if $p\gentzen p\lor q$ and $p\lor q\gentzen q$, then $p\gentzen q$. Hence, $p\gentzen q$ if{}f $p\lor q\iffgentzen q$. Furthermore, $p\lor q\iffgentzen q$ if{}f ($r\iffgentzen p\lor q$ if{}f $r\iffgentzen q$) if{}f $\igequ{p\lor q}=\igequ{q}$ if{}f $\igequ{p}\lor\igequ{q}=\igequ{q}$ if{}f $\igequ{p}\leq\igequ{q}$. Hence, $p\gentzen q\in\coQInt$ if{}f $\igequ{p}\leq\igequ{q}$ $\forall p\gentzen q\in\Seq(\mathcal{L}^{\boti,\topi}_{\cqineg})$. By \ref{coqint a7} (resp., \hyperlink{coqint r4}{\textup{r4)}}) this gives $\igequ{p}^{\eva\eva}=\igequ{\cqineg\cqineg p}\leq\igequ{p}$ (resp., if $\igequ{p}\leq\igequ{q}$, then $\igequ{q}^\eva\leq\igequ{p}^\eva$) $\forall\igequ{p},\igequ{q}\in \bar{K}$. Furthermore,  \ref{coqint a8} gives $\igequ{p}\leq\igequ{q\lor\cqineg q}=\igequ{q}\lor\igequ{\cqineg q}=\igequ{q}\lor\igequ{q}^\eva$ $\forall\igequ{p},\igequ{q}\in \bar{K}$. \ref{coqint a13} and \ref{coqint a14} give $\igequ{p}\leq1$ and $0\leq\igequ{p}$ $\forall\igequ{p}\in \bar{K}$, respectively, implying boundedness of $(\bar{K},\leq,\land,\lor,0,1)$. Distributivity of $(\bar{K},\leq,\land,\lor)$ follows from two inequalities: \ref{coqint a6} gives $\igequ{p}\land(\igequ{q}\lor\igequ{r})\leq(\igequ{p}\land\igequ{q})\lor(\igequ{p}\land\igequ{r})$ $\forall\igequ{p},\igequ{q},\igequ{r}\in \bar{K}$, while the converse inequality follows from \hyperlink{coqint r3}{\textup{r3)}} applied to ((\hyperlink{coqint r1}{\textup{r1)}} applied to \ref{coqint a3} and \ref{coqint a4}) and (\hyperlink{coqint r1}{\textup{r1)}} applied to \ref{coqint a3} and \ref{coqint a5})). As a result, the structure $(\bar{K},\leq,\land,\lor,0,1,{}^\eva)$ is a coquasiintuitionistic algebra. Let $\Chi_{\igequ{\cdot}}(p):=\igequ{p}$ $\forall p\in\Frm(\mathcal{L}^{\boti,\topi}_{\cqineg})$. Hence, $\Chi_{\igequ{\cdot}}$ is a valuation of $\Frm(\mathcal{L}^{\boti,\topi}_{\cqineg})$ in $(\bar{K},\leq,\land,\lor,0,1,{}^\eva)$. Furthermore, $p\gentzen q\in\coQInt$ if{}f $\Chi_{\igequ{\cdot}}(p)\leq\Chi_{\igequ{\cdot}}(q)$ $\forall p\gentzen q\in\Seq(\mathcal{L}^{\boti,\topi}_{\cqineg})$, since $p\gentzen q\in\coQInt$ if{}f $\igequ{p}\leq\igequ{q}$ $\forall p\gentzen q\in\Seq(\mathcal{L}^{\boti,\topi}_{\cqineg})$. Assume $p\gentzen q\notin\coQInt$. Then $p\gentzen q$ is not true under a valuation $\Chi_{\igequ{\cdot}}$ in a coquasiintuitionistic algebra $(\bar{K},\leq,\land,\lor,0,1,{}^\eva)$. Hence, there exists a coquasiintuitionistic algebra in which $p\gentzen q$ is not valid. Thus, $\coQInt$ is complete with respect to the class of coquasiintuitionistic algebras.
\end{enumerate}
\item Analogously to \ref{prop.soundness.of.coQInt 1}.
\item Follows from \ref{prop.soundness.of.coQInt 1} and \ref{prop.soundness.of.coQInt 2}. 
\end{enumerate}
\end{proof}

\begin{remark}\label{rem.logics}
\begin{enumerate}[nosep, label=(\roman*)]
\item\label{rem.logics.lattice.vakarelov} The subset of $\Seq(\mathcal{L}_0)$ (resp., $\Seq(\mathcal{L}_0)$; $\Seq(\mathcal{L}^{\boti,\topi})$; $\Seq(\mathcal{L}^{\boti,\topi})$) consisting of the axioms \ref{coqint a1}--\ref{coqint a5} (resp., \ref{coqint a1}--\ref{coqint a6}; \ref{coqint a1}--\ref{coqint a5}, \ref{coqint a13}, \ref{coqint a14}; \ref{coqint a1}--\ref{coqint a6}, \ref{coqint a13}, \ref{coqint a14}) and all sequents obtained from them by application of the rules \hyperlink{coqint r1}{r1)}--\hyperlink{coqint r3}{r3)} and the rule of substitution forms a (resp., \textit{distributive}; \textit{bounded}; \textit{bounded distributive}) \textit{lattice logic} \cite[pp. 113--114, 118]{Skordev:1969} \cite{SchulteMoenting:1973} \cite[pp. 19--22]{Goldblatt:1974} \cite[pp. 1--2, 10--12]{Vakarelov:1976} \cite[p. 587]{Vakarelov:1980} \cite[pp. 294--295, 299]{SchulteMoenting:1981} \cite[\S1.1]{Vakarelov:1989}, which is sound and complete with respect to a class of all (resp., distributive; bounded; bounded distributive) lattices \cite[Thms. 4.1, 4.5]{Orlowska:Vakarelov:2005}.  The smallest lattice logic in $\Seq(\mathcal{L}_{\cqineg})$ that is closed under \hyperlink{coqint r4}{r4)} has been introduced in \cite[Ch. III]{Vakarelov:1976}. We will call it the \emph{Vakarelov logic} and denote it by $\Vak$. It is sound and complete with respect to the class of all Vakarelov algebras (they were introduced as `negative logical algebras' in \cite[p. 57]{Vakarelov:1976}). The smallest logic in $\Seq(\mathcal{L}_{\qineg,\cqineg})$ that is closed under \hyperlink{coqint r4}{r4)} and \hyperlink{coqint r5}{r5)} will be called a \emph{bi-Vakarelov logic}, and denoted by $\biVak$. It is sound and complete with respect to the class of all bi-Vakarelov algebras. 
\item \label{rem.logics.coint.biint}
The extension $\mathcal{L}^{\boti,\topi}_{\himp}$ of $\mathcal{L}^{\boti,\topi}$ with a binary connective $\himp$ allows \cite[p. 118]{Skordev:1969} \cite[p. 330]{Vakarelov:1989} to characterise \emph{intuitionistic logic} ${\Int}$ as the smallest bounded distributive lattice logic in $\Seq(\mathcal{L}^{\boti,\topi}_{\himp})$ that contains all sequents of the form
\begin{enumerate}[nosep,label=\textup{a\arabic*)}]
\setcounter{enumii}{14}
\item $p\land(p\himp q)\gentzen q$\label{coqint a17}
\end{enumerate}
and is closed under the rule
\begin{enumerate}[nosep]
\item[\hypertarget{coqint r6}{r6)}] $p\land q\gentzen r$ if{}f $p\gentzen q\himp r$.
\end{enumerate}
Analogously, by $(\land,\lor)$-duality, the extension $\mathcal{L}^{\boti,\topi}_{\himp,\cohimpi}$ (resp., $\mathcal{L}^{\boti,\topi}_{\cohimpi}$) of $\mathcal{L}^{\boti,\topi}_{\himp}$ (resp., $\mathcal{L}^{\boti,\topi}$) with a binary connective $\cohimp$ allows \cite[p. 173]{Drobyshevich:2018} to characterise the \emph{biintuitionistic logic} ${\biInt}$ (resp., \emph{cointuitionistic logic} ${\coInt}$) as the smallest bounded distributive lattice logic in $\Seq(\mathcal{L}^{\boti,\topi}_{\himp,\cohimpi})$ (resp., $\Seq(\mathcal{L}^{\boti,\topi}_{\cohimpi})$) that contains all sequents of the form \ref{coqint a17} and \ref{coqint a18} (resp., \ref{coqint a18}) and is closed under the rules \hyperlink{coqint r6}{r6)} and \hyperlink{coqint r7}{r7)} (resp., under the rule \hyperlink{coqint r7}{r7)}), where
\begin{enumerate}[nosep,label=\textup{a\arabic*)}]
\setcounter{enumii}{15}
\item $q\gentzen p\lor(q\cohimp p)$,\label{coqint a18}
\item[\hypertarget{coqint r7}{r7)}] $p\gentzen q\lor r$ if{}f $p\cohimp q\gentzen r$.
\end{enumerate}
Equivalently, $\Int$ (resp., $\coInt$; $\biInt$) can be characterised \cite[II.2.16]{Vakarelov:1976} as the smallest bounded distributive lattice logic in $\Seq(\mathcal{L}_\qineg^{\boti})$ (resp., $\Seq(\mathcal{L}_\cqineg^{\topi})$; $\Seq(\mathcal{L}_{\qineg,\cqineg}^{\topi,\boti})$) that is closed under the rule \hyperlink{coqint r8}{r8)} (resp., \hyperlink{coqint r9}{r9)}; \hyperlink{coqint r8}{r8)} and \hyperlink{coqint r9}{r9)}), where:
\begin{enumerate}[nosep]
\item[\hypertarget{coqint r8}{r8)}] $p\land q\gentzen\bot$ if{}f $p\gentzen\qineg q$;
\item[\hypertarget{coqint r9}{r9)}] $\top\gentzen p\lor q$ if{}f $\cqineg p\gentzen q$.
\end{enumerate}
\item\label{rem.logics.moisil.rauszer} An axiomatic formulation of ${\Int}$ was first given in \cite[p. 184]{Glivenko:1929} and \cite[App.]{Heyting:1930}. Its algebraic model in terms of Heyting algebras was given in \textup{\cite[p. 161]{Ogasawara:1939}} \textup{\cite[Thm. 46]{Birkhoff:1942}} \textup{\cite[Thm. 9.5]{Certaine:1943}}. An axiomatic formulation of ${\biInt}$ has been provided independently in \cite[\S2]{Moisil:1942} and \cite[II.\S1]{Rauszer:1970} (=\cite[\S9]{Rauszer:1974}).\footnote{Equivalence of these two approaches was proved in \cite[\S1]{Drobyshevich:Odincov:Wansing:2022}. An independent formulation of $\biInt$, without algebraic semantics, was provided in \cite[\S3]{Klemke:1971}.} \cite[II.\S\S1--2]{Rauszer:1970} (=\cite[\S\S9--10]{Rauszer:1974}) established soundness and completeness of ${\biInt}$ with respect to the variety of all Skolem algebras. The historical development of $\coInt$ is more involved: its first axiomatisation as a lattice logic (implying its soundness and completeness with respect to the variety of Brouwer algebras) was provided in \cite[p. 72]{Vakarelov:1976}; independently, other early attempts of its axiomatisation appeared in \cite[p. 126]{Dummett:1976}, \cite[pp. 325--327]{Czermak:1977:Ueber} \cite[pp. 471--472]{Czermak:1977}, and \cite[pp. 121--124]{Goodman:1981}, and were completed in \cite[\S3]{Urbas:1996} and \cite[\S3]{Gore:2000}. By construction, $\Int$ is an extension of $\QInt$, $\coInt$ is an extension of $\coQInt$, and all of them are extensions of $\Vak$.
\end{enumerate}
\end{remark}

\begin{definition}
Let $\mathcal{L}^{\boti,\topi}_{\qineg,\cqineg,\himp,\cohimpi}$ (resp., $\mathcal{L}^{\boti,\topi}_{\qineg,\himp}$; $\mathcal{L}^{\boti,\topi}_{\cqineg,\himp}$; $\mathcal{L}^{\boti,\topi}_{\qineg,\cohimp}$; $\mathcal{L}^{\boti,\topi}_{\cqineg,\cohimp}$) denote the extension of $\mathcal{L}^{\boti,\topi}_{\qineg,\cqineg}$ (resp., $\mathcal{L}^{\boti,\topi}_{\qineg}$; $\mathcal{L}^{\boti,\topi}_{\cqineg}$; $\mathcal{L}^{\boti,\topi}_{\qineg}$; $\mathcal{L}^{\boti,\topi}_{\cqineg}$) with the binary connectives $\himp$ and $\cohimp$ (resp., $\himp$; $\himp$; $\cohimp$; $\cohimp$). The smallest logic that contains all sequents of the form $S$ consisting of the axioms $A$ and all sequents obtained from them by application of the rules $R$ and the rule of substitution will be called $N$, and will be denoted by $L$, where the quintuples $(S;A;R;N;L)$ are given by:
\begin{enumerate}[nosep, label=(\roman*)]
\item ($\mathcal{L}^{\boti,\topi}_{\qineg,\himp}$; an \emph{intuitionistic--quasiintuitionistic logic}; \ref{coqint a1}--\,\ref{coqint a6}, \ref{coqint a9}, \ref{coqint a10}, and \ref{coqint a12}--\,\ref{coqint a14} and \ref{coqint a17}; \hyperlink{coqint r1}{\textup{r1)}}--\,\hyperlink{coqint r3}{\textup{r3)}}, \hyperlink{coqint r5}{\textup{r5)}}, and \hyperlink{coqint r6}{\textup{r6)}}; $\QInt\otimes\Int$).
\item ($\mathcal{L}^{\boti,\topi}_{\cqineg,\himp}$; an \emph{intuitionistic--coquasiintuitionistic logic}; \ref{coqint a1}--\,\ref{coqint a8}, \ref{coqint a11}, \ref{coqint a13}, \ref{coqint a14}, and \ref{coqint a17}; \hyperlink{coqint r1}{\textup{r1)}}--\,\hyperlink{coqint r4}{\textup{r4)}}, and \hyperlink{coqint r6}{\textup{r6)}}; $\coQInt\otimes\Int$).
\item ($\mathcal{L}^{\boti,\topi}_{\qineg,\cohimp}$; \ref{coqint a1}--\,\ref{coqint a6}, \ref{coqint a9}, \ref{coqint a10}, and \ref{coqint a12}--\,\ref{coqint a14}, \ref{coqint a18}; \hyperlink{coqint r1}{\textup{r1)}}--\,\hyperlink{coqint r3}{\textup{r3)}}, \hyperlink{coqint r5}{\textup{r5)}}, and \hyperlink{coqint r7}{\textup{r7)}}; a \emph{cointuitionistic--quasiintuitionistic logic}; $\QInt\otimes\coInt$).
\item ($\mathcal{L}^{\boti,\topi}_{\cqineg,\cohimp}$; a \emph{cointuitionistic--coquasiintuitionistic logic}; \ref{coqint a1}--\,\ref{coqint a8}, \ref{coqint a11}, \ref{coqint a13}, \ref{coqint a14}, and \ref{coqint a18}; \hyperlink{coqint r1}{\textup{r1)}}--\,\hyperlink{coqint r4}{\textup{r4)}}, and \hyperlink{coqint r7}{\textup{r7)}}; $\coQInt\otimes\coInt$).
\item ($\Seq(\mathcal{L}^{\boti,\topi}_{\qineg,\cqineg,\himp,\cohimpi})$; \ref{coqint a1}--\,\ref{coqint a18}; \hyperlink{coqint r1}{\textup{r1)}}--\,\hyperlink{coqint r7}{\textup{r7)}}; an \emph{Akchurin logic}; $\biQInt\otimes\biInt$).
\end{enumerate}
\end{definition}

\begin{corollary}\label{cor.akchurin.is.sound.and.complete}
\begin{enumerate}[nosep, label=(\roman*)]
\item $\QInt\otimes\Int$ is sound and complete with respect to the class of all Heyting--quasiintuitionistic algebras;
\item $\coQInt\otimes\Int$ is sound and complete with respect to the class of all Heyting--coquasiintuitionistic algebras;
\item $\QInt\otimes\coInt$ is sound and complete with respect to the class of all Brouwer--quasiintuitionistic algebras;
\item $\coQInt\otimes\coInt$ is sound and complete with respect to the class of all Brouwer--coquasiintuitionistic algebras;
\item $\biQInt\otimes\biInt$ is sound and complete with respect to the class of all Akchurin algebras;
\item all of the above logics are extensions of $\biVak$;
\item $\biQInt\otimes\biInt$ is a common extension of $\QInt\otimes\Int$, $\coQInt\otimes\Int$, $\QInt\otimes\coInt$, and $\coQInt\otimes\coInt$.
\end{enumerate}
\end{corollary}

\begin{proof}%\ \\
%\vspace{-13pt}
\begin{enumerate}[nosep, label=(\roman*)]
\item[(v)] Follows from Proposition \ref{prop.soundness.of.coQInt}.\ref{prop.soundness.of.coQInt 3} and \cite[II.\S\S1--2]{Rauszer:1970} (=\cite[\S{}\S9--10]{Rauszer:1974}) (cf. Remark \ref{rem.logics}.\ref{rem.logics.moisil.rauszer}).
\item[(i)--(iv)] Follows from Proposition \ref{prop.soundness.of.coQInt}.\ref{prop.soundness.of.coQInt 1}--\ref{prop.soundness.of.coQInt 2} and Remark \ref{rem.logics}.\ref{rem.logics.coint.biint}--\ref{rem.logics.moisil.rauszer}.
\item[(vi)--(vii)] Follows from definitions of $\biVak$ and $\biQInt\otimes\biInt$.
\end{enumerate}
\end{proof}

\section{\label{section.spectral.presheaf.daseinisations}Spectral presheaf and daseinisations}

\begin{remark}
From now on, until the end of this paper, $L$ will denote a complete orthocomplemented lattice with an orthocomplementation $^\perp$.
\end{remark}

\begin{definition}
Let $\mathtt{C}$ be a category, and let $F,G:\mathtt{C}^{\mathrm{op}}\ra\mathtt{Set}$ be presheaves.
\begin{enumerate}[nosep, label=(\roman*)]
\item $F$ is called a \emph{subpresheaf} of $G$ if{}f, for each $A,B\in \mathrm{Obj}(\mathtt{C}^{\mathrm{op}})$ and each $f\in \mathrm{Hom}(B,A)$, $F(A)\subseteq G(A)$ and the following diagram commutes:
\[
\xymatrix@=2em{
F(A)\ar[d]_{\subseteq}\ar[r]^{F(f)\;}&
F(B)\ar[d]^{\subseteq}\\
G(A)\ar[r]^{G(f)\;}&
G(B).
}
\]
\item Expressions of the form $F(A)$ (resp., $G(A)$) will be denoted by $F_{A}$ (resp., $G_B$).
\end{enumerate}
\end{definition}

\begin{definition}\label{def: Context category}
For a complete orthocomplemented lattice $L$, the \emph{context category} $\context$ of $L$ is defined as the poset of complete boolean sublattices of $L$ ordered by inclusion, and excluding the boolean lattice $\{0,1\}$. $\context$ is interpreted as a category with the elements (resp., inclusions) of $\context$ understood as objects (resp., morphisms), identifying the notation $V\in\context$ with $V\in\Ob(\context)$. The objects of $\context$ will be called \emph{contexts}.
\end{definition}

\begin{definition}\textup{\cite[p. 36]{Cannon:2013} (=\cite[Def. 3.1]{Cannon:Doering:2018})}\label{def.spectral.presheaf}
The \emph{spectral presheaf} of $L$ is defined as the presheaf $\Sigma^{L}: (\context)^{\mathrm{op}} \to \mathtt{Set}$, acting by $B \mapsto \sps(B)$ $\forall B\in \mathrm{Ob}(\context)$ and by $(B\subseteq D) \mapsto \cdot|_{B}$ on the morphisms of $\context$, where $\cdot|_{B}:\sps(D)\ra \sps(B)$ is the restriction map $\lambda\mapsto\lambda|_{B}$ $\forall\lambda\in\sps(D)$.
\end{definition}

\begin{definition}\textup{\cite[Def. 7.1.3]{Cannon:2013} (=\cite[Def. 4.3]{Cannon:Doering:2018})}\label{def.sub.clop}
The set of \emph{closed-and-open subpresheaves} of the spectral presheaf $\Sigma^{L}$, denoted by $\Subclop(\Sigma^{L})$, is defined as the set of all subpresheaves $F$ of $\Sigma^{L}$ such that, at each $V\in\context$, $F_{V}$ is closed-and-open with respect to the topology of Proposition \ref{thm: Stone}.\ref{thm: Stone.1}.
\end{definition}

\begin{proposition}\label{thm: Sigma is complete lattice}
\textup{\cite[p. 80]{Cannon:2013} (=\cite[p. 51]{Cannon:Doering:2018})}\footnote{\label{footnote.independent.of.oml}This result has been stated under an additional assumption of orthomodularity of $(L,{}^\perp)$. However, its proof does not depend on this assumption.}
$\Subclop(\Sigma^{L})$ is a complete distributive lattice, when equipped with:
\begin{enumerate}[nosep, label=(\roman*)]
\item\label{thm: Sigma is complete lattice.0} partial order given by monomorphisms; 
\item\label{thm: Sigma is complete lattice.i} the top element given by $\Sigma^{L}$;
\item\label{thm: Sigma is complete lattice.ii} the bottom element given by the empty presheaf $\varnothing$;
\item\label{thm: Sigma is complete lattice.iii} $(\bigwedge_{i\in I}{F^i})_{V} := \mathrm{int}(\bigcap_{i\in I}({F^i})_{V})$;
\item\label{thm: Sigma is complete lattice.iv} $(\bigvee_{i\in I}{F^i})_{V} := \mathrm{cl}(\bigcup_{i\in I}({F^i})_V)$,
\end{enumerate}
for a set $I$, $\{F^i\mid i\in I\}\subseteq\Subclop(\Sigma^{L})$, $V\in \context$, and with the topological interior $(\mathrm{int})$ and topological closure $(\mathrm{cl})$ operators defined by the topology of Proposition \ref{thm: Stone}.
\end{proposition}

\begin{remark}
If the assumption of completeness of $L$ Definition \ref{def: Context category} is dropped off, then the corresponding objects in Definitions \ref{def.spectral.presheaf} and \ref{def.sub.clop} are still well defined, and the corresponding version of Proposition \ref{thm: Sigma is complete lattice} also holds. However, this is no longer so for the objects and properties introduced in the rest of the current Section.
\end{remark}

\begin{proposition}\label{prop: lattice Inner daseinisation fun}
If $(L,\leq,\land,\lor)$ is a complete sublattice of complete lattice $(K,\leq,\land,\lor)$, and
\begin{align}
\odeltabar_L(x)&:=\bigwedge\{y\in L\mid y\geq x\}\;\;\forall x\in K;\label{eqn.odeltabar.def}\\
\ideltabar_L(x)&:=\bigvee\{y\in L\mid y\leq x\}\;\;\forall x\in K,\label{eqn.ideltabar.def}
\end{align}
then:
\begin{enumerate}[nosep, label=(\roman*)]
\item\label{item: lattice Inner daseinisation fun 1} $\ideltabar_{L}$ and $\odeltabar_{L}$ are order monotone;
\item\label{item: lattice Inner daseinisation fun 2} $\ideltabar_{L}$ (resp., $\odeltabar_{L}$) preserves all meets (resp., all joins);
\item\label{item: lattice Inner daseinisation fun 3} $\bigvee \ideltabar_{L}(M) \leq \ideltabar_{L}(\bigvee M)$ and $\bigwedge \odeltabar_{L}(M) \geq \odeltabar_{L}(\bigwedge M)$ for any nonempty subset $M$ of $L$;
\item\label{item: lattice Inner daseinisation fun 5} if $(L,\leq,\land,\lor)$ and $(K,\leq,\land,\lor)$ are bounded with a bottom (resp., top) element $0$ (resp., $1$), then $\ideltabar_{L}(0)=0=\odeltabar_{L}(0)$ and $\ideltabar_{L}(1)=1=\odeltabar_{L}(1)$.
\end{enumerate}
\end{proposition}

\begin{proof}
\begin{enumerate}[nosep, label=(\roman*)]
\item Let $a,b\in L$ be such that $a\leq b$. By \eqref{eqn.ideltabar.def}, $\ideltabar_{L}(a)\in L$ and $\ideltabar_{L}(a) \leq a$. Thus, $\ideltabar_{L}(a)\leq a\leq b$. Definition of the join, $\ideltabar_{L}(b) = \bigvee\lbrace d\in L\mid d\leq b\rbrace$, and $\ideltabar_{L}(a) \leq b$ give $\ideltabar_{L}(a) \leq \ideltabar_{L}(b)$.
\item Let $M$ be a nonempty subset of $L$. By \ref{item: lattice Inner daseinisation fun 1}, $ \ideltabar_{L}(\bigwedge M)\leq \ideltabar_{L}(a)$ $\forall a\in M$. Hence $ \ideltabar_{L}(\bigwedge M) \leq \bigwedge \ideltabar_{L}(M)$. Thus $\bigwedge \ideltabar_{L}(M) \leq \bigwedge M$, by $ \ideltabar_{L}(a) \leq a$ $\forall a\in M$. Since, by definition, $\ideltabar_{L}(\bigwedge M)$ is the largest element in $L$ smaller than $\bigwedge M$, this implies $  \bigwedge \ideltabar_{L}(M)\leq \ideltabar_{L}(\bigwedge M)$, hence $\ideltabar_{L}(\bigwedge M) = \bigwedge \ideltabar_{L}(M)$. 
\item If $a\in M$, then $a\leq \bigvee M$ and, by \ref{item: lattice Inner daseinisation fun 1}, $\ideltabar_{L}(a) \leq \ideltabar_{L}(\bigvee M)$ $\forall a\in M$. Thus, $\bigvee \ideltabar_{L}(M) \leq \ideltabar_{L}(\bigvee M)$.  
\item $L$ is a sublattice, hence $\ideltabar_{L}(0) = \bigvee \{b\in L\mid b\leq 0\} = 0$, so $0\in L$ (resp., $\ideltabar_{L}(1) = \bigvee \{b\in L\mid b\leq 1\} = 1$, so $1\in L$).
\end{enumerate}
The results for $\odeltabar_{L}$ follow by $(\land,\lor)$-duality.
\end{proof}

\begin{definition}\label{def: Inner daseinisation}
Let $a\in L$. The \emph{inner daseinisation} is defined as a map $\idelta: L \ni a \mapsto \idelta(a)\in \Subclop(\Sigma^{L})$, with the functor $\idelta(a):(\context)^{\mathrm{op}} \to \mathtt{Set},$ acting by $V\mapsto (\idelta(a))_{V}\coloneqq \stone_{V}\left(\ideltabar_{V}(a)\right)$ $\forall V\in \mathrm{Ob}(\context)$, where
\begin{equation}
\ideltabar_{V}: L \ni a \mapsto \ideltabar_{V}(a)=\bigvee\lbrace b\in V\mid b\leq a\rbrace \in V\;\;\forall V\in \context,
\label{eqn.inner.daseinisation}
\end{equation}
and acting by $(V\subseteq W) \mapsto \cdot|_{V}$ on morphisms of $\context$, where $|_{V}:\sps(W) \ra \sps(V)$ is the restriction map $\lambda\mapsto\lambda|_{V}$ $\forall\lambda\in\sps(W)$.
\end{definition}

\begin{proposition}\label{prop: Inner Daseinisation}
$\idelta$ is order monotone, preserves the bottom and top elements and all meets, is injective, and satisfies $\bigvee \idelta(M)\leq \idelta(\bigvee M)$ for all nonempty subsets $M$ of $L$.
\end{proposition}
\begin{proof}
Order monotonicity (resp., preservation of meets; $\bigvee \idelta(M)\leq \idelta(\bigvee M)$ for all nonempty subsets $M$ of $L$; preservation of the bottom and the top elements) follows from Proposition \ref{prop: lattice Inner daseinisation fun}.\ref{item: lattice Inner daseinisation fun 1} (resp., \ref{prop: lattice Inner daseinisation fun}.\ref{item: lattice Inner daseinisation fun 2}; \ref{prop: lattice Inner daseinisation fun}.\ref{item: lattice Inner daseinisation fun 3}; \ref{prop: lattice Inner daseinisation fun}.\ref{item: lattice Inner daseinisation fun 5}) and the fact that $\stone_{V}$ is a boolean algebra isomorphism $\forall V\in \context$. Consider $\minbool{c}\coloneqq \{0,c,c^{\perp},1\}$ $\forall c\in L$. $\minbool{c}$ is a finite boolean subalgebra of $L$ $\forall c\in L$, so$\minbool{c}\in \context$ $\forall c\in L$. Assume that $a,b\in L$ are comparable (i.e. $a\leq b$ or $b\leq a$) and satisfy $a\neq b$. Let $a< b$. Then $\ideltabar_{\minbool{b}}(a) = 0$ and $\ideltabar_{\minbool{b}}(b)=b$ by \eqref{eqn.ideltabar.def}. Since $\stone$ is a boolean algebra isomorphism, Definition \ref{def: Inner daseinisation} implies $(\idelta(a))_{\minbool{b}} = \stone_{\minbool{b}}(\ideltabar_{\minbool{b}}(a)) = \stone_{\minbool{b}}(0) \neq \stone_{\minbool{b}}(b) = \stone_{\minbool{b}}(\idelta_{\minbool{b}}(b)) = (\idelta(b))_{\minbool{b}}$, so  $\idelta(a)\neq\idelta(b)$. The same argument holds for $b<a$, as well as if $a$ and $b$ are not comparable (i.e. if $a\not < b$ and $b\not < a$). Thus, $\idelta$ is injective.
\end{proof}

\begin{proposition}\label{thm: idelta adjunction}
There is an adjunction $\iepsilon\dashv\idelta$, where
\begin{equation}
\iepsilon :\Subclop(\Sigma^{L}) \ni S \mapsto\iepsilon(S) \coloneqq \bigwedge\{a\in L\mid  S\leq\idelta(a) \}\in L.
\end{equation}
In particular, $\iepsilon$ is order monotone, preserves the bottom and the top elements and all joins, is surjective, satisfies $\iepsilon(\bigwedge M) \leq \bigwedge \iepsilon(M)$ for any nonempty subset $M$ of $\Subclop(\Sigma^{L})$, and $\iepsilon\circ \idelta = \mathrm{id}_{L}$.
\end{proposition}

\begin{proof}
Order monotonicity, adjointness, and preservation of joins follows from the adjoint functor theorem for posets \cite[Ex. \S3.J]{Freyd:1964}. For any nonempty subset $M$ of $\Subclop(\Sigma^{L})$, order monotonicity and $\bigwedge M \leq S$ $\forall S\in M$ implies $\iepsilon(\bigwedge M) \leq \iepsilon(S)$ $\forall S\in M$, so $\iepsilon(\bigwedge M) \leq \bigwedge \iepsilon(M)$. For $\iepsilon\circ \idelta = \mathrm{id}_{L}$, if $a\in L$, then $(\iepsilon\circ \idelta)(a) = \iepsilon(\idelta(a)) = \bigwedge \{b\in L\mid \idelta(a) \leq \idelta(b)\} = a$, since $\idelta(a)$ is the largest element in $\idelta(L)$ smaller or equal than $\idelta(a)$, and $\idelta(b) = \idelta(a)$ implies $b=a$ by injectivity of $\idelta$. $\iepsilon\circ \idelta = \mathrm{id}_{L}$ implies surjectivity of $\iepsilon$ and, together with preservation of the bottom and top elements by $\idelta$ (Proposition \ref{prop: Inner Daseinisation}), implies preservation of the bottom and top elements by $\iepsilon$, since $0 = (\iepsilon\circ \idelta)(0) = \iepsilon(\varnothing)$ and $1 = (\iepsilon\circ \idelta)(1) = \iepsilon (\Sigma^{L})$.
\end{proof}

\begin{definition}\textup{\cite[p. 86]{Cannon:2013} (=\cite[pp. 54--55]{Cannon:Doering:2018})}\label{def: Outer daseinisation}
Let $a\in L$. The \emph{outer daseinisation} is defined as a map $\odelta: L \ni a \mapsto \odelta(a)\in \Subclop(\Sigma^{L})$, with functor $\odelta(a):(\context)^{\mathrm{op}} \to \mathtt{Set},$ acting by $V\mapsto (\odelta(a))_{V}\coloneqq \stone_{V}\left(\odeltabar_{V}(a)\right)$ $\forall V\in \mathrm{Ob}(\context)$, where
\begin{equation}
\odeltabar_{V}: L \ni a \mapsto \odeltabar_{V}(a)=\bigwedge\lbrace b\in V\mid a\leq b\rbrace \in V\;\;\forall V\in \context,
\label{eqn.outer.daseinisation}
\end{equation}
and acting by $(V\subseteq W) \mapsto\cdot|_{V}$ on morphisms of $\context$, where $\cdot|_{V}:\sps(W)\ra\sps(V)$ is the restriction map $\lambda\mapsto\lambda|_{V}$ $\forall\lambda\in\sps(W)$.
\end{definition}

\begin{proposition}\label{thm: Daseinisation}
$\odelta$ is order monotone, preserves the bottom and top elements and all joins, is injective, and satisfies $\odelta(\bigwedge M) \leq \bigwedge\odelta(M)$ for all nonempty subsets $M$ of $L$.
\end{proposition}

\begin{proof}
Follows from Proposition \ref{prop: Inner Daseinisation} by $(\wedge,\vee)$-duality.
\end{proof}

\begin{proposition}\label{thm: delta is right adjoint}
There is an adjunction $\odelta \dashv \oepsilon$, where
\begin{equation}
\oepsilon :\Subclop(\Sigma^{L}) \ni S \mapsto \oepsilon(S) \coloneqq \bigvee\{a\in L\mid \odelta(a) \leq S\}\in L
\end{equation}
is order monotone, preserves the bottom and top elements and all meets, is surjective, satisfies $\oepsilon(\bigvee M) \geq  \bigvee \oepsilon(M)$ for all nonempty subsets $M$ of $\Subclop(\Sigma^{L})$, and $\oepsilon\circ \odelta = \mathrm{id}_{L}$.
\end{proposition}

\begin{proof}
Follows from Proposition \ref{thm: idelta adjunction} by $(\wedge,\vee)$-duality.
\end{proof}

\begin{remark}
Proposition \ref{thm: Daseinisation} (resp., \ref{thm: delta is right adjoint}) was proved in \textup{\cite[Lem. 7.3.2]{Cannon:2013} (=\cite[Lem. 4.11]{Cannon:Doering:2018})} (resp.,  \textup{\cite[Lem. 7.4.2]{Cannon:2013} (=\cite[Lem. 4.13]{Cannon:Doering:2018})}). These lemmas assume orthomodularity of $(L,{}^\perp)$, but their proofs not depend on this assumption.
\end{remark}

\section{\label{section.akchurin.spectral.presheaf}Akchurin algebra of the spectral presheaf}

\begin{lemma}\textup{\cite[Thm. 4.2]{Eva:2015}}\footnote{Except of \ref{item: Properties of * 13}--\ref{item: Properties of * 12}.}$^{\mathrm{,}\,}$\textup{\footref{footnote.independent.of.oml}}\label{lem: Properties of *}
The map
\begin{equation}
(\cdot)^{\eva}:\Subclop(\Sigma^{L}) \ni S \mapsto S^{\eva}:=\odelta((\oepsilon(S))^{\perp})\in \Subclop(\Sigma^{L})
\end{equation}
satisfies $\forall S,T\in \Subclop(\Sigma^{L})$ $\forall a\in L$:
\begin{enumerate}[nosep, label=(\roman*)]
\begin{minipage}{0.5\linewidth}
\item $S\lor S^{\eva}=\Sigma^{L};$\label{item: Properties of * 1}
\item $S^{\eva \eva}=\odelta(\oepsilon(S))\leq S$;\label{item: Properties of * 2}
\item $S^{\eva\eva\eva}=S^{\eva}$;\label{item: Properties of * 3}
\item  $S\land S^{\eva}\geq\varnothing$;\label{item: Properties of * 4}
\item $(S\land T)^{\eva}= S^{\eva}\lor T^{\eva}$;\label{item: Properties of * 5}
\item $(S\lor T)^{\eva} \leq S^{\eva}\land T^{\eva}$;\label{item: Properties of * 6}
\item $\oepsilon(S) \lor \oepsilon(S^{\eva}) = 1$;\label{item: Properties of * 7}
\end{minipage}
\begin{minipage}{0.5\linewidth}
\item $\oepsilon(S) \land \oepsilon(S^{\eva}) = 0$;\label{item: Properties of * 8}
\item if $S\leq T$, then $S^{\eva}\geq T^{\eva}$;\label{item: Properties of * 9}
\item $\left(\odelta(a)\right)^{\eva} = \left(\odelta(a^{\perp})\right)$;\label{item: Properties of * 13}
\item $\oepsilon(S^{\eva}) = \left(\oepsilon(S)\right)^{\perp}$;\label{item: Properties of * 14}
\item $\left(\odelta(a)\right)^{\eva \eva} = \odelta(a)$;\label{item: Properties of * 10}
\item $(\Sigma^{L})^{\eva}=\varnothing$;\label{item: Properties of * 11}
\item $\varnothing^{\eva}=\Sigma^{L}$.\label{item: Properties of * 12}
\end{minipage}
\end{enumerate}
\end{lemma}

\begin{proof}
Let $S,T\in\Subclop(\Sigma^L)$ and $a\in L$.
\begin{enumerate}[nosep,label=(\roman*)]
\item[(ii)] $S^{\eva \eva} = \odelta((\oepsilon(\odelta((\oepsilon(S))^{\perp}) ))^{\perp}) = \odelta(((\oepsilon(S))^{\perp})^{\perp}) = \odelta(\oepsilon(S) ),$
due to $\oepsilon\circ\odelta=\mathrm{id}_{L}$, by Proposition \ref{thm: delta is right adjoint}. By Proposition \ref{thm: delta is right adjoint}, $\odelta \dashv \oepsilon$, hence $\odelta(\oepsilon(S))\leq S$.
\item[(i)] Using Proposition \ref{thm: Daseinisation} and \ref{item: Properties of * 2}, $S\lor S^{\eva} = S\lor \odelta((\oepsilon(S))^{\perp} ) \geq \odelta(\oepsilon(S) ) \lor \odelta((\oepsilon(S))^{\perp} )=\odelta( \oepsilon(S) \lor (\oepsilon(S))^{\perp} )= \odelta( 1) = \Sigma^{L}$, so $S\lor S^{\eva}=\Sigma^{L}$.
\setcounter{enumi}{2}
\item $S^{\eva\eva\eva} = (\odelta(\oepsilon(S)))^{\eva}= \odelta((\oepsilon( \odelta(\oepsilon(S))) )^{\perp} )= \odelta((\oepsilon(S))^{\perp}) = S^{\eva}$.
\item $\varnothing$ is the bottom element of $\Subclop(\Sigma^{L})$ by Proposition \ref{thm: Sigma is complete lattice}.
\item $(S\land T)^{\eva} = \odelta(( \oepsilon( S\land T ))^{\perp})= \odelta(( \oepsilon( S)\land \oepsilon(T ))^{\perp})= \odelta( (\oepsilon( S))^{\perp} \lor (\oepsilon(T ))^{\perp})= \odelta( (\oepsilon( S))^{\perp}) \lor \odelta((\oepsilon(T ))^{\perp}) = S^{\eva}\lor T^{\eva}$, by the properties of $\odelta$ (resp., $\oepsilon$) given in Proposition \ref{thm: Daseinisation} (resp., \ref{thm: delta is right adjoint}).
\item $(S\lor T)^{\eva} = \odelta(( \oepsilon( S\lor T ))^{\perp})\leq  \odelta(( \oepsilon( S)\lor \oepsilon(T ))^{\perp}) = \odelta((\oepsilon( S))^{\perp}\land (\oepsilon(T ))^{\perp}) \leq  \odelta( (\oepsilon( S))^{\perp}) \land \odelta((\oepsilon(T ))^{\perp}) = S^{\eva}\land T^{\eva} $, by the properties of $\odelta$ (resp. $\oepsilon$) given in Proposition \ref{thm: Daseinisation} (resp., \ref{thm: delta is right adjoint}).
\item $\oepsilon(S) \lor \oepsilon(S^{\eva}) = \oepsilon(S) \lor \oepsilon(\odelta((\oepsilon(S))^{\perp} ) )= \oepsilon(S) \lor (\oepsilon(S))^{\perp}=1$, by Proposition \ref{thm: delta is right adjoint}.
\item $\oepsilon(S) \land \oepsilon(S^{\eva}) = \oepsilon(S) \land \oepsilon(\odelta((\oepsilon(S))^{\perp} ) ) = \oepsilon(S) \land (\oepsilon(S))^{\perp} =0$, by Proposition \ref{thm: delta is right adjoint}.
\item $\odelta$ and $\oepsilon$ are order monotone and $^\perp$ is order reversing, hence $S\leq T$  implies $\oepsilon(S)\leq \oepsilon(T)$ implies $(\oepsilon(S))^{\perp}\geq (\oepsilon(T))^{\perp}$ implies $\odelta((\oepsilon(S))^{\perp})\geq \odelta((\oepsilon(T))^{\perp})$, which is equivalent to $S^{\eva} \geq T^{\eva}$.
\item $(\odelta(a))^{\eva} = \odelta((\oepsilon(\odelta(a)))^{\perp}) = \odelta(a^{\perp})$, by Proposition \ref{thm: delta is right adjoint}.
\item $\oepsilon(S^{\eva}) = \oepsilon\left(\odelta((\oepsilon(S))^{\perp}) \right) = (\oepsilon(S))^{\perp}$, by Proposition \ref{thm: delta is right adjoint}.
\item $(\odelta(a))^{\eva \eva} = \odelta(a^{\perp\perp}) = \odelta(a),$ by \ref{item: Properties of * 13} and \ref{item: Properties of * 14}.
\item Due to Propositions \ref{thm: Daseinisation} and \ref{thm: delta is right adjoint}, $(\Sigma^{L})^{\eva} = \odelta((\oepsilon(\Sigma^{L}))^{\perp}) = \odelta(1^{\perp} ) = \odelta(0) = \varnothing.$
\item Analogously to the proof of \ref{item: Properties of * 11}.
\end{enumerate}
\end{proof}

\begin{remark}
Lemmas \ref{lem: Properties of *}.\ref{item: Properties of * 3} and \ref{lem: Properties of *}.\ref{item: Properties of * 5}--\ref{item: Properties of * 6} are consequences of Lemmas \ref{lemm.bullet.negation.properties}.\ref{lemm.bullet.negation.properties 1}--\ref{lemm.bullet.negation.properties 3}, \ref{lem: Properties of *}.\ref{item: Properties of * 1}--\ref{item: Properties of * 2}, and \ref{lem: Properties of *}.\ref{item: Properties of * 9}. Lemmas \ref{lem: Properties of *}.\ref{item: Properties of * 11}--\ref{item: Properties of * 12} are a consequence of Lemma \ref{lemm.bullet.negation.properties}.\ref{lemm.bullet.negation.properties 45}, Proposition \ref{thm: Sigma is complete lattice}.\ref{thm: Sigma is complete lattice.i}--\ref{thm: Sigma is complete lattice.ii}, and Lemmas \ref{lem: Properties of *}.\ref{item: Properties of * 1}--\ref{item: Properties of * 2}, and \ref{lem: Properties of *}.\ref{item: Properties of * 9}.
\end{remark}

\begin{corollary}\label{cor.subclop.is.coquasiint}
$(\Subclop(\Sigma^{L}),\leq,\land,\lor,\varnothing,\Sigma^{L},{}^{\eva})$ is a complete coquasiintuitionistic algebra.
\end{corollary}

\begin{proof}
Follows from Lemmas \ref{lem: Properties of *}.\ref{item: Properties of * 1}--\ref{item: Properties of * 2} and \ref{lem: Properties of *}.\ref{item: Properties of * 9}  and Proposition \ref{thm: Sigma is complete lattice}.
\end{proof}

\begin{lemma}\label{lemma.subclop.coeva.properties}
The map 
\begin{equation}
(\cdot)^\coeva:\Subclop(\Sigma^L)\ni S\mapsto S^\coeva:=\idelta((\iepsilon(S))^\perp)\in\Subclop(\Sigma^L)
\end{equation}
satisfies $\forall S,T\in\Subclop(\Sigma^L)$ $\forall a\in L$:
\begin{enumerate}[nosep, label=(\roman*)]
\begin{minipage}{0.5\linewidth}
\item $(\idelta(a))^{\coeva} = \idelta(a^{\perp})$;\label{lemma.subclop.coeva.properties.0a}
\item $\iepsilon(S^{\coeva}) = (\iepsilon(S))^{\perp}$;\label{lemma.subclop.coeva.properties.0b}
\item $S^{\coupint} = \idelta(\iepsilon(S))$;\label{lemma.subclop.coeva.properties.0c}
\item $S\land S^\coeva=\varnothing$;\label{lemma.subclop.coeva.properties.1}
\item $S^{\coeva\coeva}\geq S$;\label{lemma.subclop.coeva.properties.2}
\item $S^{\coeva\coeva\coeva}=S^\coeva$;\label{lemma.subclop.coeva.properties.3}
\item $S\land S^\coeva\leq\Sigma^L$;\label{lemma.subclop.coeva.properties.4}
\end{minipage}
\begin{minipage}{0.5\linewidth}
\item $(S\lor T)^\coeva=S^\coeva\land T^\coeva$;\label{lemma.subclop.coeva.properties.5}
\item $(S\land T)^\coeva\geq S^\coeva\lor T^\coeva$;\label{lemma.subclop.coeva.properties.6}
\item $\iepsilon(S)\land\iepsilon(S^\coeva)=0$;\label{lemma.subclop.coeva.properties.7}
\item $\iepsilon(S)\lor\iepsilon(S^\coeva)=1$;\label{lemma.subclop.coeva.properties.8}
\item if $S\leq T$, then $T^\coeva\leq S^\coeva$;\label{lemma.subclop.coeva.properties.9}
\item $(\idelta(a))^{\coeva\coeva}=\idelta(a)$;\label{lemma.subclop.coeva.properties.10}
\item $(\Sigma^L)^\coeva=\varnothing$;\label{lemma.subclop.coeva.properties.11}
\item $\varnothing^\coeva=\Sigma^L$.\label{lemma.subclop.coeva.properties.12}
\end{minipage}
\end{enumerate}
\end{lemma}

\begin{proof}
Follows from Lemma \ref{lem: Properties of *} by $(\land,\lor)$-duality.
\end{proof}

\begin{corollary}\label{cor.coeva.qint}
\begin{enumerate}[nosep, label=(\roman*)]
\item $(\Subclop(\Sigma^L),\leq,\land,\lor,\varnothing,\Sigma^L,{}^\coeva)$ is a complete quasiintuitionistic algebra.\label{cor.coeva.qint.1}
\item $(\Subclop(\Sigma^L),\leq,\land,\lor,\varnothing,\Sigma^L,{}^\coeva,{}^\eva)$ is a complete biquasiintuitionistic algebra.\label{cor.coeva.qint.2}
\end{enumerate}
\end{corollary}

\begin{proof}
\begin{enumerate}[nosep, label=(\roman*)]
\item Follows from Proposition \ref{thm: Sigma is complete lattice} and Lemmas \ref{lemma.subclop.coeva.properties}.\ref{lemma.subclop.coeva.properties.1}--\ref{lemma.subclop.coeva.properties.2} and \ref{lemma.subclop.coeva.properties}.\ref{lemma.subclop.coeva.properties.9}.
\item Follows from \ref{cor.coeva.qint.1} and Corollary \ref{cor.subclop.is.coquasiint}.
\end{enumerate}
\end{proof}

\begin{proposition}\textup{\cite[Prop. 7.2.6]{Cannon:2013} (=\cite[Prop. 4.9]{Cannon:Doering:2018})}\textup{\footref{footnote.independent.of.oml}}\label{cor.heyting.brouwer.skolem}
\begin{enumerate}[nosep, label=(\roman*)]
\item $(\Subclop(\Sigma^{L}),\leq,\land,\lor,\varnothing,\Sigma^{L},\protect\himp,\protect\hnot)$ is a complete Heyting algebra;\label{cor.heyting.brouwer.skolem.1}
\item $(\Subclop(\Sigma^{L}),\leq,\land,\lor,\varnothing,\Sigma^{L},\protect\cohimp,\protect\cohnot)$ is a complete Brouwer algebra;\label{cor.heyting.brouwer.skolem.2}
\item $(\Subclop(\Sigma^{L}),\leq,\land,\lor,\varnothing,\Sigma^{L},\protect\himp,\protect\hnot,\protect\cohimp,\protect\cohnot)$ is a complete Skolem algebra.\label{cor.heyting.brouwer.skolem.3}
\end{enumerate}
\end{proposition}

\begin{corollary}\label{cor.complete.bqi.akchurin}
\begin{enumerate}[nosep, label=(\roman*)]
\item\label{cor.complete.bqi.akchurin.i} $(\Subclop(\Sigma^{L}),\leq,\land,\lor,\varnothing,\Sigma^{L},\protect\himp,\protect\hnot,{}^\coeva)$ is a complete Heyting--quasiintuitionistic algebra, and a sound model of $\QInt\otimes\Int$;
\item\label{cor.complete.bqi.akchurin.ib} $(\Subclop(\Sigma^{L}),\leq,\land,\lor,\varnothing,\Sigma^{L},\protect\himp,\protect\hnot,{}^\eva)$ is a complete Heyting--coquasiintuitionistic algebra, and a sound model of $\coQInt\otimes\Int$;
\item\label{cor.complete.bqi.akchurin.ii} $(\Subclop(\Sigma^{L}),\leq,\land,\lor,\varnothing,\Sigma^{L},\protect\cohimp,\protect\cohnot,{}^\eva)$ is a complete Brouwer--coquasiintuitionistic algebra, and a sound model of $\coQInt\otimes\coInt$;
\item\label{cor.complete.bqi.akchurin.iib} $(\Subclop(\Sigma^{L}),\leq,\land,\lor,\varnothing,\Sigma^{L},\protect\cohimp,\protect\cohnot,{}^\coeva)$ is a complete Brouwer--quasiintuitionistic algebra, and a sound mo\-del of $\coQInt\otimes\Int$;
\item\label{cor.complete.bqi.akchurin.iii}\sloppy $(\Subclop(\Sigma^{L}),\leq,\land,\lor,\varnothing,\Sigma^{L},\protect\himp,\protect\hnot,\protect\cohimp,\protect\cohnot,{}^\coeva,{}^\eva)$ is a complete Akchurin algebra, and a sound model of $\biQInt\otimes\biInt$.
\end{enumerate}
\end{corollary}

\begin{proof}
Follows from Proposition \ref{cor.heyting.brouwer.skolem} and Corollaries \ref{cor.akchurin.is.sound.and.complete}, \ref{cor.subclop.is.coquasiint}, and \ref{cor.coeva.qint}.
\end{proof}

\section[Internal lattices of $\Subclop(\Sigma^L)$]{\label{section.internal.lattice.subclop}Internal lattices of $(\Subclop(\Sigma^{L}),\leq,\land,\lor,\varnothing,\Sigma^{L},\protect\himp,\protect\hnot,\protect\cohimp,\protect\cohnot,{}^\coeva,{}^\eva)$}

\begin{proposition}\label{prop: delta**}
Let $\subclopevaeva\coloneqq \{S^{\upint} \mid S\in \Subclop(\Sigma^{L})\} \subseteq \Subclop(\Sigma^{L})$ be equipped with the restriction of the partial order $\leq$ on $\Subclop(\Sigma^{L})$ to $\subclopevaeva$ and the restriction of ${}^\eva$ on $\Subclop(\Sigma^{L})$ to $\subclopevaeva$, denoted without change of symbol. Then $(\subclopevaeva,\leq,\curlywedge,\lor,\varnothing,\Sigma^L,{}^\eva)$ is an orthocomplemented lattice such that:
\begin{enumerate}[nosep, label=(\roman*)]
\item\label{prop: delta**.i} the join $\lor$ on $\subclopevaeva$ is given by the join $\lor$ of $\Subclop(\Sigma^{L})$ restricted to $\subclopevaeva$;
\item\label{prop: delta**.ii} the meet $\curlywedge$ on $\subclopevaeva$ is given by $\alpha \curlywedge \beta \coloneqq \left(\alpha \land \beta\right)^{\eva \eva} $ $\forall \alpha,\beta \in \subclopevaeva$;
\item\label{prop: delta**.iii} the bottom (resp., top) element of $\subclopevaeva$ is given by $\varnothing$ (resp., $\Sigma^L$);
\item\label{prop: delta**.iv} the orthocomplementation is given by ${}^{\eva}$.
\end{enumerate}
\end{proposition}
\begin{proof}
Follows from Proposition \ref{prop: CQI Internal} and Corollary \ref{cor.subclop.is.coquasiint}.
\end{proof}

\begin{corollary}\label{cor: delta(L)** is lattice}
$(\subclopevaeva,\leq,\curlywedge,\lor,\varnothing,\Sigma^L,{}^\eva)$ is a complete lattice.
\end{corollary}

\begin{proof}
Follows from Propositions \ref{prop: Internal completeness} and \ref{thm: Sigma is complete lattice}.
\end{proof}

\begin{proposition}\textup{\cite[p. 91]{Cannon:2013} (=\cite[p. 61]{Cannon:Doering:2018})}\textup{\footref{footnote.independent.of.oml}}\label{prop.oepsilon.equivalence.class}
$S\approx_{\oepsilon} T$ if{}f $\oepsilon(S) = \oepsilon(T)$ (resp., $S\approx_{\iepsilon} T$ if{}f $\iepsilon(S) = \iepsilon(T)$) $\forall S,T\in \Subclop(\Sigma^{L})$ defines an equivalence relation $\approx_{\oepsilon}$ (resp., $\approx_{\iepsilon}$) on $\Subclop(\Sigma^{L})$. 
\end{proposition}
\begin{proof}
$\approx_{\oepsilon}$ (resp., $\approx_{\iepsilon}$) is defined in terms of equality in $\Subclop(\Sigma^{L})$.
\end{proof}

\begin{definition}
The equivalence class of $S$ with respect to $\approx_{\oepsilon}$ (resp., $\approx_{\iepsilon}$) will be denoted by $\oepsilonequc{S}$ (resp., $\iepsilonequc{S}$) $\forall S\in \Subclop(\Sigma^{L})$.
\end{definition}

\begin{proposition} \textup{\cite[Lem. 7.5.1]{Cannon:2013}} \textup{\cite[p. 166]{Eva:2015}} \textup{\cite[Thm. 4.19]{Cannon:Doering:2018}}\textup{\footref{footnote.independent.of.oml}} \label{prop: L iso e(L)}
For $\oepsilonequc{\Subclop(\Sigma^{L})}\coloneqq \{\oepsilonequc{ S} \mid \forall S\in \Subclop(\Sigma^{L})\}$, $(\oepsilonequc{\Subclop(\Sigma^{L})},\leq^{\oepsilon},\land^{\oepsilon},\lor^{\oepsilon},\oepsilonequc{\varnothing},\oepsilonequc{\Sigma^{L}},{}^{\eva})$ is an orthocomplemented lattice such that, for $S,T\in \Subclop(\Sigma^{L})$ and $M\subseteq\Subclop(\Sigma^{L})$:
\begin{enumerate}[nosep, label=(\roman*)]
\item $\oepsilonequc{S}\leq^{\oepsilon} \oepsilonequc{T}$ if{f} $\oepsilon(S) \leq \oepsilon(T)$;
\item $\bigwedge^{\oepsilon}\oepsilonequc{M} \coloneqq \oepsilonequc{\bigwedge M}$;
\item $\bigvee^{\oepsilon}\oepsilonequc{M} \coloneqq \bigwedge^{\oepsilon}\{\nu\in \oepsilonequc{\Subclop(\Sigma^{L})}\mid \mu \leq^{\oepsilon} \nu$ $\forall \mu\in \oepsilonequc{M} \}$;
\item $(\oepsilonequc{S})^{\eva} \coloneqq \oepsilonequc{S^{\eva}}$.
\end{enumerate}
\end{proposition}

\begin{proposition}\label{prop: L iso d(a)**}
\sloppy $(\subclopevaeva,\leq,\curlywedge,\lor,\varnothing,\Sigma^L,{}^\eva)\iso (\oepsilonequc{\Subclop(\Sigma^{L})},\leq^{\oepsilon},\land^{\oepsilon},\lor^{\oepsilon},\oepsilonequc{\varnothing},\oepsilonequc{\Sigma^{L}},{}^{\eva})\iso(L,\leq,\land,\lor,0,1,{}^\perp)$.
\end{proposition}

\begin{proof}
\begin{enumerate}[nosep, label=\alph*)]
\item For $(\subclopevaeva,\leq,\curlywedge,\lor,\varnothing,\Sigma^L,{}^\eva)\iso (L,\leq,\land,\lor,0,1,{}^\perp)$ consider $\odelta: L \to \odelta(L)$. This map is a bijection, since $\odelta$ is an injection by Proposition \ref{thm: Daseinisation}. By Lemma \ref{lem: Properties of *}.\ref{item: Properties of * 2}, $S^{\upint} = \odelta(\oepsilon(S))$ $\forall S\in \Subclop(\Sigma^{L})$. Thus $\subclopevaeva\subseteq \odelta(L)$. By Lemma \ref{lem: Properties of *}.\ref{item: Properties of * 10}, $(\odelta(a))^{\upint} = \odelta(a)$ $\forall a\in L$, so $\odelta(L)\subseteq \subclopevaeva$. Thus $\subclopevaeva=\odelta(L)$ as sets, so $\odelta: L \to \evaintlatticeext$ is also a bijection. This isomorphism of sets extends to an isomorphism of orthocomplemented lattices, since it preserves the lattice structure of $L$:
\begin{enumerate}[nosep, label=(\roman*)]
\item $\odelta(a^{\perp}) =  \left(\odelta(a)\right)^{\eva} $ $\forall a\in L$, by Proposition \ref{thm: Daseinisation};
\item $\odelta(a\lor b)   =\odelta(a)\lor \odelta(b)$ $\forall a,b\in L$, by Proposition \ref{thm: Daseinisation};
\item $\odelta(0) = \varnothing$ (resp., $\odelta(1) = \Sigma^{L}$), by Proposition \ref{thm: Daseinisation};
\item $\odelta(a\land b)  = \odelta((a^{\perp}\lor b^{\perp})^{\perp}) = (\odelta(a^{\perp})\lor \odelta(b^{\perp}))^{\eva} = ((\odelta(a))^{\eva} \lor (\odelta(b))^{\eva})^{\eva} = (\odelta(a) \land \odelta(b))^{\upint} = \odelta(a)\curlywedge \odelta(b)$ $\forall a,b\in L$, since ${}^{\perp}$ is an orthocomplementation, and by Lemma \ref{lem: Properties of *}.
\end{enumerate}
\item\sloppy ${(L,\leq,\land,\lor,0,1,{}^\perp)\iso (\oepsilonequc{\Subclop(\Sigma^{L})},\leq^{\oepsilon},\land^{\oepsilon},\lor^{\oepsilon},\oepsilonequc{\varnothing},\oepsilonequc{\Sigma^{L}},{}^{\eva})}$ follows from \textup{\cite[Thm. 4.19]{Cannon:Doering:2018}}\textup{\footref{footnote.independent.of.oml}}. 
\end{enumerate}
\end{proof}

\begin{corollary}
Let $\subclopcoevacoeva\coloneqq \{S^{\coupint} \mid S\in \Subclop(\Sigma^{L})\} \subseteq \Subclop(\Sigma^{L})$ be equipped with the restrictions of $\leq$ and $^\coeva$ on $\Subclop(\Sigma^L)$ to $\subclopcoevacoeva$, denoted without change of symbol. Then $(\subclopcoevacoeva,\leq,\land,\curlyvee,\varnothing,\Sigma^L,{}^\coeva)$ is an orthocomplemented lattice such that:
\begin{enumerate}[nosep, label=(\roman*)]
\item the meet $\land$ on $\subclopcoevacoeva$ is given by the meet $\land$ of $\Subclop(\Sigma^L)$ restricted to $\subclopcoevacoeva$;
\item the join $\curlyvee$ on $\subclopcoevacoeva$ is given by $\alpha\curlyvee\beta:=(\alpha\lor\beta)^{\coeva\coeva}$ $\forall\alpha,\beta\in\subclopcoevacoeva$;
\item the bottom (resp., top) element of $\subclopcoevacoeva$ is given by $\varnothing$ (resp., $\Sigma^L$);
\item the orthocomplementation is given by ${}^\coeva$.
\end{enumerate}
\end{corollary}

\begin{proof}
Follows Corollary \ref{cor.coeva.qint}.\ref{cor.coeva.qint.1} and Corollary \ref{cor: QI Internal}.
\end{proof}

\begin{corollary}\label{cor: Inner delta(L)** is lattice}
$(\subclopcoevacoeva,\leq,\land,\curlyvee,\varnothing,\Sigma^L,{}^\coeva)$ is a complete lattice.
\end{corollary}

\begin{proof}
Follows from Corollary \ref{cor: QI Internal completeness} and Proposition \ref{thm: Sigma is complete lattice}.
\end{proof}

\begin{proposition}\label{prop: Inner L iso d(a)**}
For $S,T\in \Subclop(\Sigma^{L})$, and $M\subseteq\Subclop(\Sigma^{L})$, let:
\begin{enumerate}[nosep, label=\alph*)]
\item $\iepsilonequc{\Subclop(\Sigma^{L})}\coloneqq \{\iepsilonequc{ S} \mid \forall S\in \Subclop(\Sigma^{L})\}$;
\item $\iepsilonequc{S}\leq^{\iepsilon} \iepsilonequc{T}$ if{f} $\iepsilon(S) \leq \iepsilon(T)$;
\item $\bigwedge^{\iepsilon}\iepsilonequc{M} \coloneqq \iepsilonequc{\bigwedge M}$;
\item $\bigvee^{\iepsilon}\iepsilonequc{M} \coloneqq \bigwedge^{\iepsilon}\{\nu\in \iepsilonequc{\Subclop(\Sigma^{L})}\mid \mu \leq^{\iepsilon} \nu$ $\forall \mu\in \iepsilonequc{M} \}$;
\item $(\iepsilonequc{S})^{\coeva} \coloneqq \iepsilonequc{S^{\coeva}}$. 
\end{enumerate}
Then:
\begin{enumerate}[nosep, label=(\roman*)]
\item\label{item: Inner L iso d(a)** 1} $(\iepsilonequc{\Subclop(\Sigma^{L})},\leq^{\iepsilon},\land^{\iepsilon},\lor^{\iepsilon},\iepsilonequc{\varnothing},\iepsilonequc{\Sigma^{L}},{}^{\coeva})$ is an orthocomplemented lattice;
\item\label{item: Inner L iso d(a)** 2} \sloppy $(\subclopcoevacoeva,\leq,\land,\curlyvee,\varnothing,\Sigma^L,{}^\coeva)\iso (\iepsilonequc{\Subclop(\Sigma^{L})},\leq^{\iepsilon},\land^{\iepsilon},\lor^{\iepsilon},\iepsilonequc{\varnothing},\iepsilonequc{\Sigma^{L}},{}^{\coeva})\iso (L,\leq,\land,\lor,0,1,{}^\perp)\iso(\subclopevaeva,\leq,\curlywedge,\lor,\varnothing,\Sigma^L,{}^\eva)\iso (\oepsilonequc{\Subclop(\Sigma^{L})},\leq^{\oepsilon},\land^{\oepsilon},\lor^{\oepsilon},\oepsilonequc{\varnothing},\oepsilonequc{\Sigma^{L}},{}^{\eva})$.
\end{enumerate}
\end{proposition}

\begin{proof}
\ref{item: Inner L iso d(a)** 1} (resp., \ref{item: Inner L iso d(a)** 2}) follows from Proposition \ref{prop: L iso e(L)} (resp., \ref{prop: L iso d(a)**}) by $(\land,\lor)$-duality.
\end{proof}

\begin{proposition}\label{cor.orthomodularity.from.subclop}
The following conditions are equivalent:
\begin{enumerate}[nosep, label=(\roman*)]
\item\label{cor.orthomodularity.from.subclop.i}$\forall x\in\Subclop(\Sigma^L)\setminus\{\varnothing,\Sigma^L\}$ $\nexists y\in\Subclop(\Sigma^L)\setminus\{\varnothing,\Sigma^L\}$ $\st$ $x^\eva\leq y$ and $(x\land y)^\eva=1$;
\item\label{cor.orthomodularity.from.subclop.ii}$\forall x\in\Subclop(\Sigma^L)\setminus\{\varnothing,\Sigma^L\}$ $\nexists y\in\Subclop(\Sigma^L)\setminus\{\varnothing,\Sigma^L\}$ $\st$ $y\leq x^\coeva$ and $(x\lor y)^\coeva=0$;
\item\label{cor.orthomodularity.from.subclop.iii}$(\subclopevaeva,\leq,\curlywedge,\lor,\varnothing,\Sigma^L,{}^\eva)$ is an orthomodular lattice;
\item\label{cor.orthomodularity.from.subclop.iv}$(\subclopcoevacoeva,\leq,\land,\curlyvee,\varnothing,\Sigma^L,{}^\coeva)$ is an orthomodular lattice;
\item\label{cor.orthomodularity.from.subclop.v}$(L,\leq,\land,\lor,0,1,{}^\perp)$ is an orthomodular lattice.
\end{enumerate}
\end{proposition}

\begin{proof}
\ref{cor.orthomodularity.from.subclop.iii} if{}f \ref{cor.orthomodularity.from.subclop.iv} if{}f \ref{cor.orthomodularity.from.subclop.v} follows from Proposition \ref{prop: Inner L iso d(a)**}.\ref{item: Inner L iso d(a)** 2}. \ref{cor.orthomodularity.from.subclop.i} if{}f \ref{cor.orthomodularity.from.subclop.iii} (resp., \ref{cor.orthomodularity.from.subclop.ii} if{}f \ref{cor.orthomodularity.from.subclop.iv}) follows from Proposition \ref{prop: Precqi internal oml} (resp., Corollary \ref{cor.quasiintuitionistic.internal.orthomodular.lattice}).
\end{proof}

\section{\label{section.nogo.for.relevance}$(\Subclop(\Sigma^L),\leq,\land,\lor,\varnothing,\Sigma,{}^\eva)$ is not a model of a relevance logic}

\begin{proposition}\label{prop.subclop.demorgan.iff.L.minimal}
\sloppy $(\Subclop(\Sigma^L),\leq,\land,\lor,{}^\eva)$ (resp., $(\Subclop(\Sigma^L),\leq,\land,\lor,{}^\coeva)$) is a De Morgan algebra if{}f $(\Subclop(\Sigma^L),\leq,\land,\lor,\varnothing,\Sigma^L,{}^\eva)$ (resp., $(\Subclop(\Sigma^L),\leq,\land,\lor,\varnothing,\Sigma^L,{}^\eva)$) is a boolean lattice.
\end{proposition}

\begin{proof}
Assume that $(\Subclop(\Sigma^L),\leq,\land,\lor,{}^\eva)$ is a De Morgan algebra. Then ${}^{\eva}$ is an involution, so $S^{\upint} = S$ $\forall S\in \Subclop(\Sigma^{L})$ by \ref{def: negation 2} and \ref{def: negation 3}, $(\Sigma^{L})^{\eva} = \varnothing$ by Lemma \ref{lemm.bullet.negation.properties}.\ref{lemm.bullet.negation.properties 7}, and $S^{\eva}\land T^{\eva} = (S\lor T)^{\eva}$ by Lemma \ref{lemm.bullet.negation.properties}.\ref{lemm.bullet.negation.properties 4}. Thus $S\land S^{\eva} = S^{\upint} \land S^{\eva} = (S^{\eva} \lor S)^{\eva} = (\Sigma^{L})^{\eva} = \varnothing$, since $S\lor S^{\eva} = \Sigma^{L}$ $\forall S\in \Subclop(\Sigma^{L})$ by Lemma \ref{lem: Properties of *}.\ref{item: Properties of * 1}. Thus $(\Subclop(\Sigma^L),\leq,\land,\lor,\varnothing,\Sigma^L,{}^\eva)$ is a boolean lattice. Assume that $(\Subclop(\Sigma^L),\leq,\land,\lor,\varnothing,\Sigma^L,{}^\eva)$ is a boolean lattice. Then $(\Subclop(\Sigma^L),\leq,\land,\lor,{}^\eva)$ is a De Morgan algebra, since $(\Subclop(\Sigma^L),\leq,\land,\Sigma^L,{}^\eva)$ is distributive and ${}^{\eva}$ is an orthocomplementation. \sloppy The claim for $(\Subclop(\Sigma^L),\leq,\land,\lor,{}^\coeva)$ follows by $(\land,\lor)$-duality.
\end{proof}

\begin{corollary}\label{cor.nogo.for.relevance.logics.in.Subclop}
Neither $(\Subclop(\Sigma^L),\leq,\land,\lor,\varnothing,\Sigma^L,{}^\coeva)$ nor $(\Subclop(\Sigma^L),\leq,\land,\lor,\varnothing,\Sigma^L,{}^\eva)$ is a sound algebraic model of any of the relevance logic systems in the range from $\mathbf{B^{\orlov,t}}$ to $\mathbf{R^{t}}$ (hence, including $\mathbf{DJ^{\orlov,t}}$, $\mathbf{DK^{\orlov,t}}$, and $\mathbf{DL^{\orlov,t}}$). The same holds also for the entailment logic $\mathbf{E^{\orlov,t}}$, the relevant logic $\mathbf{R}$, the first degree entailment logic $\mathbf{FDE}$, and the extensions of all these systems equipped with universal and existential quantifiers (denoted by adding the postfix $\mathbf{Q}$).
\end{corollary}

\begin{proof}
All of these logical systems exhibit a De Morgan negation, cf. \cite[p. 457]{Anderson:Belnap:1958} for $\mathbf{E}$, \textup{\cite[p. 287]{Routley:Meyer:Plumwood:Brady:1982}}\footnote{An earlier variant of $\mathbf{B^t}$, additionally featuring an axioms $p \lor\!\!\sim\!\!p$ was introduced in \cite[p. 415, p. 422]{Meyer:Routley:1972}. (Both of these systems differ substantially from the `basic theory of implication' that was introduced under the same symbol in \cite[VII.\S2]{Belnap:1959}, and was later studied also in \cite[\S{}VIII]{Dunn:1966}.)} for $\mathbf{B}$, \textup{\cite[pp. 355--356]{Brady:1984}} for $\mathbf{DJ}$, \textup{\cite[p. 76]{Routley:1977}} for $\mathbf{DK}$; \textup{\cite[p. 7]{Routley:Meyer:1976}}\footnote{In \cite[p. 7]{Routley:Meyer:1976} the axioms of $\mathbf{DL}$ contain also an axiom $p_{0} \land\!\!\sim\!\!p_{0}$ with a fixed propositional constant $p_0$, which was later dropped off in \cite[p. 290]{Routley:Meyer:Plumwood:Brady:1982} and \cite[pp. 355--356]{Brady:1984}. On the other hand, the axioms of $\mathbf{DL}$ in \cite[p. 290]{Routley:Meyer:Plumwood:Brady:1982} contain also $p \lor\!\!\sim\!\!p$, that is neither in \cite[p. 7]{Routley:Meyer:1976} nor in \cite[pp. 355--356]{Brady:1984}.} for $\mathbf{DL}$, \textup{\cite[p. 2]{Belnap:1967}} for $\mathbf{R}$. (The index $^\mathbf{t}$ indicates adding a fixed propositional constant $t$, and assuming the additional axioms \cite[p. 198]{Meyer:1966} $t$ and $t\triangleright(p\triangleright p)$, while the index $^\orlov$ indicates adding a binary connective $\orlov$, and assuming an additional rule \cite[Eqn. (6a)]{Orlov:1928}: $(p\orlov q)\triangleright r$ if{}f $p\triangleright(q\triangleright r)$, where $\triangleright$ denotes an implication. These additional structures, in the same way as the additional axioms that are introducing existential and universal quantifiers (cf., e.g., \cite[pp. 357--358]{Brady:1984}), do not influence the elementary properties of negation in the respective systems.) When restricted to a reduct $(K,\leq,\land,\lor,\sim)$ given by a lattice $(K,\leq,\land,\lor)$ equipped with a negation $\sim:K\ra K$, the sound and complete algebraic models of these logics exhibit the common structure of a De Morgan algebra, cf. \cite[\S{}\S{}VII.2--4]{Dunn:1966} for $\mathbf{FDE}$, \cite[\S{}\S{}X.3--X.4]{Dunn:1966} for $\mathbf{R^t}$, \cite[Thm. 5]{Meyer:Routley:1972}\footnote{Cf. \cite[Thms. 9.2.2, 9.3.1]{Sylvan:Meyer:Brady:Mortensen:Plumwood:2003} for more explicit proof.} for a range of systems between $\mathbf{B^{\orlov,t}}$ and $\mathbf{R^t}$, \cite[Thm. (p. 466)]{Maksimova:1973:implication} for $\mathbf{E^t}$, \cite[Thms. 1, 2, 4]{Meyer:Dunn:Leblanc:1974} for $\mathbf{RQ^t}$, and \cite[Thms. 14, 15, 17]{Font:Rodriguez:1990} for $\mathbf{R}$. By Proposition \ref{prop.subclop.demorgan.iff.L.minimal}, $(\Subclop(\Sigma^L),\leq,\land,\lor,{}^\coeva)$ (resp., $(\Subclop(\Sigma^L),\leq,\land,\lor,{}^\eva)$) is a De Morgan algebra if{f} it is a boolean algebra. However, a boolean algebra is not paraconsistent (i.e. $p\land\!\!\sim\!\!p=0$ $\forall p\in K$), while all of the above logical systems are explicitly paraconsistent.
\end{proof}

\section*{Acknowledgments}

BE: I thank Tobias Fritz and Tim Netzer for their support and for providing the means to continue my scientific pursuits in this uncertain time. This project was partially funded by the grant 10.55776/\-P36684 of the Austrian Science Fund FWF. RPK: I thank Aleksandr K. Guc for a major inspiration with \guillemotleft{Siberian toposes}\guillemotright{} \cite{Guc:1999} (which included also \cite{Akchurin:1985}), Cecilia Flori, Chris Isham, Marie La Palme Reyes, and Gonzalo Reyes for their hospitality and discussions, Micha{\l} Heller for hospitality, discussions, and an invitation to Krak\'{o}w that has triggered this research, and Mariusz Stopa for some discussions as well. This work was partially funded by 2021/42/A/ST2/00356 grant of Polish National Center of Science, IMPULZ IM-2023-79 and VEGA 2/0164/25 grants of Slovak Academy of Sciences, APVV-22-0570 grant of Slovak Research and Development Agency, and FellowQUTE 2024-02 program of Slovak National Center for Quantum Technologies. Our join\footnote{Strictly speaking, it is a meet, not a join.} thanks go to Guram Bezhanishvili, Shahn Majid, Agnieszka Nawara, Greg Restall, and Vladimir~L. Vasyukov for their help in obtaining copies of a few less accessible texts.

\section*{References}
\addcontentsline{toc}{section}{References}
\begin{spacing}{0.85}
{\small
All Cyrillic Russian and Bulgarian names and words were transliterated from original using the following system (which is bijective due to the lack of {\fontencoding{T2A}\selectfont ыа} and {\fontencoding{T2A}\selectfont ыу} combinations): {\fontencoding{T2A}\selectfont ц} = c, {\fontencoding{T2A}\selectfont ч} = ch, {\fontencoding{T2A}\selectfont х} = kh, {\fontencoding{T2A}\selectfont ж} = zh, {\fontencoding{T2A}\selectfont ш} = sh, {\fontencoding{T2A}\selectfont щ} = \v{s}, {\fontencoding{T2A}\selectfont и} = i, {\fontencoding{T2A}\selectfont й} = \u{\i}, i = \={\i}, {\fontencoding{T2A}\selectfont ы} = y, {\fontencoding{T2A}\selectfont ю} = yu, {\fontencoding{T2A}\selectfont я} = ya, {\fontencoding{T2A}\selectfont ё} = \"{e}, {\fontencoding{T2A}\selectfont э} = \`{e}, {\fontencoding{T2A}\selectfont ъ} = `, {\fontencoding{T2A}\selectfont ь} = ', and analogously for capitalised letters, with an exception of {\fontencoding{T2A}\selectfont Х} = H at the beginnings of words. The symbol * in front of a bibliographic item indicates that none of us have seen this work. Unless explicitly stated otherwise, all citations refer to the first editions of the corresponding texts in their original language. Chinese Mandarin names are (nonbijectively) romanised from the original using p\={\i}ny\={\i}n.}
\end{spacing}
{\small \begingroup
\raggedright
\renewcommand\refname{

\endgroup        
}

\end{document}